\documentclass[a4paper,12pt]{article}
\usepackage{latexsym}
\usepackage{mathrsfs}
\usepackage{amsfonts,amsmath,amssymb}
\usepackage{indentfirst}
\usepackage{subeqnarray}
\usepackage{verbatim}
\usepackage{graphicx}
\usepackage{epstopdf}
\usepackage{algorithm}
\usepackage{algorithmic}
\usepackage{subfigure}
\usepackage{color}
\usepackage{multirow}
\usepackage{float}

\allowdisplaybreaks[4]

\newtheorem{theorem}{Theorem}[section]
\newtheorem{lemma}[theorem]{Lemma}
\newtheorem{assumption}[theorem]{Assumption}

\newtheorem{definition}[theorem]{Definition}
\newtheorem{example}[theorem]{Example}

\newtheorem{remark}[theorem]{Remark}

\numberwithin{equation}{section}
\newenvironment{proof}[1][Proof]{\textbf{#1.} }
{\ \rule{0.75em}{0.75em}\smallskip}

\begin{document}

\begin{center}
\Large\bf $\ell_{1}^{2}-\eta\ell_{2}^{2}$ sparsity regularization for nonlinear ill-posed problems
\end{center}

\begin{center}
Long Li\footnotemark[1] \quad and \quad Liang Ding\footnotemark[2]$^{,*}$

\footnotetext[1]{Department of Mathematics, Northeast Forestry University, Harbin 150040, China;
e-mail: {\tt 15146259835@nef u.edu.cn}.}
\footnotetext[2]{Department of Mathematics, Northeast Forestry University, Harbin 150040, China;
e-mail: {\tt dl@nefu.edu.cn}. }

\renewcommand{\thefootnote}{\fnsymbol{footnote}}
\footnotetext[1]{Corresponding author.}
\renewcommand{\thefootnote}{\arabic{footnote}}
\end{center}

\medskip
\begin{quote}
{\bf Abstract.} In this study, we investigate the $\left\|\cdot\right\|_{\ell_{1}}^{2}-\eta\left\|\cdot\right\|_{\ell_{2}}^{2}$ sparsity regularization with $0< \eta\leq 1$, in the context of nonlinear ill-posed inverse problems. We focus on the examination of the well-posedness associated with this regularization approach. Notably, the case where $\eta=1$ presents weaker theoretical outcomes than $0< \eta<1$, primarily due to the absence of coercivity and the Radon-Riesz property associated with the regularization term. Under specific conditions pertaining to the nonlinearity of the operator $F$, we establish that every minimizer of the $\left\|\cdot\right\|_{\ell_{1}}^{2}-\eta\left\|\cdot\right\|_{\ell_{2}}^{2}$ regularization exhibits sparsity. Moreover, for the case where $0<\eta<1$, we demonstrate convergence rates of $\mathcal{O}\left(\delta^{1/2}\right)$ and $\mathcal{O}\left(\delta\right)$ for the regularized solution, concerning a sparse exact solution, under differing yet widely accepted conditions related to the nonlinearity of $F$. Additionally, we present the iterative half variation algorithm as an effective method for addressing the $\left\|\cdot\right\|_{\ell_{1}}^{2}-\eta\left\|\cdot\right\|_{\ell_{2}}^{2}$ regularization in the domain of nonlinear ill-posed equations. Numerical results provided corroborate the effectiveness of the proposed methodology.

\end{quote}

\smallskip
{\bf Keywords.}  Sparsity regularization, nonlinear inverse problem, $\ell_{1}^{2}-\eta\ell_{2}^{2}$ regularization, non-convex, iterative half variation algorithm

\section{Introduction}\label{sec1}

\par The investigation of the non-convex $\left\|\cdot\right\|_{\ell_{1}}- \eta\left\|\cdot\right\|_{\ell_{2}}$ $\left(0<\eta\leq 1\right)$ regularization has garnered significant interest in the field of sparse recovery in recent years, as indicated in \cite{DH19, LCGN20, LY18, YL15} and related references. This approach serves as an alternative to the $\ell_{p}$-norm $\left(0\leq p<1\right)$ and offers a notable advantage: it approximates the $\ell_{0}$-norm well while providing a simpler computational structure. Unfortunately, challenges arise due to the non-differentiability of $\eta\left\|x\right\|_{\ell_{2}}$ at $0$. This complicates the convergence analysis of iterative algorithms, necessitating additional conditions to ensure that the solutions obtained are non-zero. Furthermore, the term $\eta x/\left\|x\right\|_{\ell_{2}}$, stemming from $\eta\left\|x\right\|_{\ell_{2}}$ within the iterative algorithm, can impact the algorithm's convergence rate. Motivated by the proximal operator of $\left\|\cdot\right\|_{\ell_{1}}^{2}$ discussed in \cite{B17}, \cite{LD25} introduced a non-convex $\left\|\cdot\right\|_{\ell_{1}}^{2}-\eta\left\|\cdot\right\|_{\ell_{2}}^{2}$ regularization to mitigate the aforementioned issues associated with $\eta\left\|x\right\| _{\ell_{2}}$. In this paper, we further explore the potential of $\left\|\cdot\right\|_{\ell_{1}}^{2}-\eta\left\|\cdot\right\|_{\ell_{2}}^{2}$ regularization method for addressing nonlinear ill-posed operator equations that admit sparse solutions. Additionally, we analyze the well-posedness of the regularization in the specific case where $\eta=1$.

\par We aim to address an ill-posed operator equation of the form
\begin{equation}\label{equ1.1}
     F\left(x\right)=y,
\end{equation}
where $x$ is sparse, $F:\ell_{2}\rightarrow Y$ is a weakly sequentially closed nonlinear operator mapping between $\ell_{2}$ and a Hilbert space $Y$ with norms $\left\|\cdot\right\|_{\ell_{2}}$ and $\left\|\cdot\right\|_{Y}$, respectively. Throughout this paper, we let $\left<\cdot,\cdot\right>$ denote the inner product in the $\ell_{2}$ space, and $e_{i}=(\underbrace{0,\cdots,0,1}_{i},0,\cdots)$, $i\geq 1$. The exact data $y$ and the observed data $y^{\delta}$ satisfy $\left\|y^{\delta}-y\right\|_{Y}\leq \delta$, where $\delta> 0$ represents the noise level.

\subsection{Relative works}\label{sec1.1}

\par The first systematic theoretical analysis of sparse regularization for addressing ill-posed inverse problems was initiated by Daubechies et al.\ in 2004 \cite{DDD04}. They proposed $\ell_{p}$ ($1\leq p\leq 2$) sparse regularization as a solution for such problems and established the convergence rate of the iterative soft-thresholding algorithm. Following this, many researchers directed their efforts towards the regularization properties and iterative frameworks for solving linear ill-posed problems, as detailed in \cite{F10, SGGHL09}, with subsequent extensions to nonlinear ill-posed problems. A significant body of work has emerged concerning the regularization properties and minimizers of sparse regularization for nonlinear ill-posed issues; refer to \cite{JM12, RT05, RT06, TB10}. However, the above literature primarily addresses the convex case where $p\geq 1$. In contrast, additional conditions are necessary for the non-convex scenario where $0\leq p<1$ to analyze the regularization properties and convergence rates. Some insights into the regularization properties and convergence rates of non-convex regularization for $0\leq p<1$ can be found in \cite{G09, WLMC13, Z09}. Therefore, the established results regarding the regularization properties of convex sparse regularization can be leveraged to examine the original non-convex sparse regularization. To minimize of $\ell_{p}$ sparse regularization with $1\leq p\leq 2$, several numerical algorithms have been developed to tackle linear ill-posed inverse problems, including alternating direction method of multipliers (ADMM) \cite{WYZ19}, iteratively reweighted least squares (IRLS) \cite{XZWCL10}, and iterative hard thresholding algorithm (IHTA) \cite{BD09}. However, these numerical techniques cannot be directly applied to nonlinear ill-posed inverse problems.

\par The non-convex $\ell_{1}-\ell_{2}$ regularization is highly effective at promoting sparsity and enhancing reconstruction accuracy, presenting notable advantages over $\ell_{p}$ regularization. For the linear sparse inverse problem, a significant amount of research has been dedicated to $\ell_{1}-\ell_{2}$ regularization. \cite{YL15} characterized $\ell_{1}-\ell_{2}$ as a non-convex yet Lipschitz continuous metric, proposing a minimization algorithm that converges to stationary points meeting the first-order optimality conditions. Building on this foundation, \cite{LY18} derived the proximal operator for the $\ell_{1}-\ell_{2}$ metric and effectively integrated it with the forward-backward splitting (FBS) technique and the alternating direction method of multipliers (ADMM), facilitating efficient numerical solutions. Furthermore, \cite{LNN20} applied $\ell_{1}-\ell_{2}$ regularization to image restoration with impulse noise, demonstrating superior results compared to conventional methodologies. From a theoretical perspective, \cite{DH19} provided a comprehensive analysis of the well-posedness for the non-convex $\ell_{1}-\eta\ell_{2}$ regularization and developed a ST-$\left(\alpha\ell_{1}-\beta\ell_{2}\right)$ algorithm for addressing. However, research on nonlinear sparse inverse problems remains limited. \cite{BLR15} introduced iterative algorithms for minimizing non-smooth, non-convex functionals, including Tikhonov functionals with nonlinear operators. \cite{DH24} extends the ST-$\left(\alpha\ell_{1}-\beta\ell_{2}\right)$ algorithm, employing iterative soft thresholding to tackle non-convex $\ell_{1}-\eta\ell_{2}$ regularization, which effectively addresses nonlinear sparse regularization problems and yields promising results.

\par Compared to the $\ell_{1}-\eta\ell_{2}$ regularization, the $\ell_{1}^{2}-\eta\ell_{2}^{2}$ regularization is also a non-convex technique, which the $\ell_{2}^{2}$ term is differentiable and does not yield the $\eta x/\left\|x\right\|_{\ell_{2}}$ term. However, the proximal operator for $\ell_{1}^{2}$ differs from the traditional soft thresholding operator. The specific form of the proximal operator for $\ell_{1}^{2}$ is discussed in \cite{B17}, while \cite{LD25} provides an analysis of its regularization properties and the corresponding solution algorithms for the linear case. Consequently, we aim to extend the $\ell_{1}^{2}-\eta\ell_{2}^{2}$ regularization method to the nonlinear scenario.

\subsection{Contribution and organization}\label{sec1.2}

\par In this paper, we address the nonlinear ill-posed inverse problem denoted as \eqref{equ1.1} using the following regularization method:
\begin{equation}\label{equ2.4}
    \min_{x\in\ell_{2}}\left\{\mathcal{J}_{\alpha,\beta}^{\delta}\left(x\right)= \frac{1}{q} \left\|F\left(x\right)-y^{\delta}\right\|_{Y}^{q}+ \mathcal{R}_{\alpha,\beta}(x_{n})\right\},
\end{equation}
where $q\geq 1$ and
\begin{equation*}
    \mathcal{R}_{\alpha,\beta}(x):=\alpha\left\|x\right\|_{\ell_{1}}^{2}- \beta\left\|x\right\|_{\ell_{2}}^{2},\quad \alpha \geq\beta >0.
\end{equation*}
Let $\eta=\beta/\alpha$, and problem \eqref{equ2.4} can be reformulated as follows:
\begin{equation}\label{equ1.2}
    \min_{x\in\ell_{2}}\left\{\mathcal{J}_{\alpha,\eta}^{\delta}\left(x\right)= \frac{1}{q}\left\|F\left(x\right)-y^{\delta}\right\|_{Y}^{q}+ \alpha\mathcal{R}_{\eta}\left(x\right)\right\},
\end{equation}
where $q\geq 1$ and 
\begin{equation}\label{equ1.3}
    \mathcal{R}_{\eta}\left(x\right):=\left\|x\right\|_{\ell_{1}}^{2} -\eta\left\|x\right\|_{\ell_{2}}^{2},\quad 0<\eta\leq 1.
\end{equation}
We will investigate the well-posedness of the problem \eqref{equ1.2}. For the case $0<\eta <1$, we demonstrate the existence, stability, and convergence of regularized solutions under the condition that the nonlinear operator $F$ is weakly sequentially closed. The numerical results reported in \cite{LD25} indicate satisfactory outcomes can be achieved when $\eta=1$. Actually, $\mathcal{R}_{\eta}$  behaves increasingly like a constant multiple of the $\ell_{0}$-norm as $\eta\rightarrow 1$. Therefore, we analyze the properties of $\mathcal{R}_{\eta}$ when $\eta=1$, although the well-posedness results for this case are not as strong as those for $0<\eta <1$. For the case $0<\eta <1$, we identify the convergence rate under an appropriate source condition. To analyze convergence rates, we impose restrictions on the nonlinearity of the operator $F$. Typically, these restrictions are used to bound the critical term $\left<F'\left(x^{\dagger}\right)\left(x-x^{\dagger}\right),\omega_{i}\right>$ when deriving convergence rate results. Under two commonly adopted conditions regarding the nonlinearity of $F$,  we obtain convergence rates of $\mathcal{O}\left(\delta^{1/2}\right)$ and $\mathcal{O}\left(\delta\right)$ for the regularized solution in the $\ell_{2}$-norm, respectively.

\par For the minimization problem \eqref{equ1.2}, \cite{DH19} proposes an iterative soft thresholding algorithm based on the generalized conditional gradient method (GCGM) to solve non-convex $\ell_{1}-\eta\ell_{2}$ sparse regularization problems. However, the GCGM requires solving two separate optimization subproblems to determine both the descent direction and step size during the iterative process. However, convergence cannot be theoretically guaranteed under a fixed step size strategy. Therefore, inspired by the proximal operator of $\left\|x\right\|_{\ell_{1}}^{2}$, we propose a proximal gradient method incorporating a fixed step size scheme for solving non-convex $\ell_{1}^{2}-\eta\ell_{2}^{2}$ sparse regularization in the context of nonlinear inverse problems. Specifically, for the case where $q=2$, we reformulate the functional  $\mathcal{J}_{\alpha,\eta}^{\delta}\left(x\right)$ in \eqref{equ1.2} as follow
\begin{equation*}
    \mathcal{J}_{\alpha,\eta}^{\delta}\left(x\right)=f\left(x\right)+g\left(x\right),
\end{equation*} 
where 
\begin{equation*}
    f\left(x\right)=\frac{1}{2}\left\|F\left(x\right)- y^{\delta}\right\|_{\ell_{2}}^{2}-\alpha\eta\left\|x\right\|_{\ell_{2}}^{2} \quad {\rm and}\quad g\left(x\right)= \alpha\left\|x\right\|_{\ell_{1}}^{2}.
\end{equation*}
Furthermore, we establish that the iterative half variation algorithm converges if the nonlinear operator $F$ satisfies the corresponding conditions so that the functions $f\left(x\right)$ and $g\left(x\right)$ meet Assumption \ref{ass4.1}.

\par The remainder of the paper is structured as follows. Section \ref{sec2} analyzes the well-posedness of the $\left\|\cdot\right\|_{\ell_{1}}^{2}- \eta\left\|\cdot\right\|_{\ell_{2}}^{2}$ $\left(0<\eta\leq 1\right)$ regularization. In Section \ref{sec3}, we derive the convergence rates in the $\ell_{2}$-norm under an appropriate source condition, alongside two commonly acknowledged conditions regarding the nonlinearity of $F$. Section \ref{sec4} presents the iterative half variation algorithm based on the proximal gradient method and discusses its convergence properties. Finally, we provide several numerical experiments in Section \ref{sec5}.

\section{Well-posedness of regularization problem}\label{sec2}

\par In this section, we analyze the well-posedness of the regularization method, focusing on the existence, stability, and convergence of regularized solutions. For the case where $\eta=1$, the lack of coercivity and the Radon-Riesz property leads to a well-posedness result weaker than that achieved when $0<\eta<1$.

\par Let us denote a general minimizer of the functional $\mathcal{J}_{\alpha,\eta}^{\delta}\left(x\right)$ by $x_{\alpha,\eta}^{\delta}$, stated as follows:
\begin{equation}\label{equ2.1}
    x_{\alpha,\eta}^{\delta}\in\arg\min_{x\in\ell_{2}}\left\{\mathcal{J}_{\alpha,\eta}^{\delta}\left(x\right) =\frac{1}{q}\left\|F\left(x\right)-y^{\delta}\right\|_{Y}^{q}+\alpha\mathcal{R}_{\eta} \left(x\right)\right\}.
\end{equation}

\begin{definition}\label{def2.1}
    An element $x^{\dagger}\in\ell_{2}$ is defined as an $\mathcal{R}_{\eta}$-minimum solution to the problem \eqref{equ1.1} if
    \begin{equation*}
        x^{\dagger}\in\arg\min_{x\in\ell_{2}}\left\{\mathcal{R}_{\eta}\left(x\right)\mid F\left(x\right)=y\right\}.
    \end{equation*}
\end{definition}

\begin{definition}\label{def2.2}
    An element $x\in\ell_{2}$ is referred to as sparse if ${\rm supp}\left(x\right):=\left\{i\in\mathbb{N} \mid x_{i}\neq 0 \right\}$ is finite, where $x_{i}$ represents the $i$th component of $x$.
\end{definition}

\par To characterize sparsity, following the approach in \cite{DDD04}, we define the index set
\begin{equation}\label{equ2.2}
    I\left(x^{\dagger}\right)=\left\{i\in\mathbb{N}\mid x_{i}^{\dagger}\neq0\right\},
\end{equation}
where $x_{i}^{\dagger}$ is the $i$th component of $x^{\dagger}$. Subsequently, we present a result regarding the non-negativity of $\mathcal{R}_{\eta}\left(x\right)$.

\subsection{The case $0<\eta<1$}\label{sec2.1}

\par We first review some essential properties of $\mathcal{R}_{\eta}\left(x\right)$ $\left(0<\eta <1\right)$, which serve as crucial tools in analyzing the well-posedness of the regularization, cf. \cite{LD25} for proofs.

\begin{lemma}\label{lem2.4}
    The functional $\mathcal{R}_{\eta}\left(x\right)$ $\left(0<\eta <1\right)$ possesses the following properties:\\
    (1) (Coercivity) For $x\in\ell_{2}$, $\left\|x\right\|_{\ell_{2}}\rightarrow \infty$ implies $\mathcal{R}_{\eta}\left(x\right)\rightarrow\infty$.\\
    (2) (Weak lower semi-continuity) If $x_{n}\rightharpoonup x$ in $\ell_{2}$ and $\left\{\mathcal{R}_{\eta}\left(x_{n}\right)\right\}$ is bounded, then 
    \begin{equation*}
        \lim\inf_{n}\mathcal{R}_{\eta}\left(x_{n}\right)\geq\mathcal{R}_{\eta}\left(x\right).
    \end{equation*}
    (3) (Radon-Riesz property) If $x_{n}\rightharpoonup x$ in $\ell_{2}$ and $\mathcal{R}_{\eta}\left(x_{n}\right)\rightarrow\mathcal{R}_{\eta}\left(x\right)$, then $\left\|x_{n}-x\right\|_{\ell_{2}}\rightarrow 0$.
\end{lemma}

\begin{lemma}
    Assume the sequence $\left\{\left\|y_{n}\right\|_{Y}\right\}$ is bounded in $Y$. For a given $M>0$, let $x_{n}\in\ell_{2}$, $n=1,2,\cdots$ and consider
    \begin{equation}\label{equ2.3}
        \frac{1}{q}\left\|F\left(x_{n}\right)-y_{n}\right\|_{Y}^{q}+\alpha\mathcal{R}_{\eta} \left(x_{n}\right)\leq M.
    \end{equation}
    Then, there exist an element $x\in\ell_{2}$ and a subsequence $\left\{x_{n_{k}}\right\}$ of $\left\{x_{n}\right\}$ such that $x_{n_{k}}\rightharpoonup x$ and $F\left(x_{n_{k}}\right)\rightharpoonup F\left(x\right)$.
\end{lemma}

\begin{proof}
    From \eqref{equ2.3}, it follows that $\left\{\mathcal{R}_{\eta}\left(x_{n}\right)\right\}$ is bounded. Consequently, by the coercivity of $\mathcal{R}_{\eta}\left(x\right)$, we deduce that $\left\{\left\|x_{n}\right\|_{\ell_{2}}\right\}$ is also bounded. Furthermore, since $\left\{\left\|y_{n}\right\|_{Y}\right\}$ is bounded, it follows that $\left\{\left\|F\left(x_{n}\right)\right\|_{Y}\right\}$ is bounded as well. Hence, there exists a subsequence $\left\{x_{n_{k}}\right\}$ of $\left\{x_{n}\right\}$, along with $x\in\ell_{2}$ and $y\in Y$, such that
    \begin{equation*}
        x_{n_{k}}\rightharpoonup x\ {\rm in}\ \ell_{2}, \quad F\left(x_{n_{k}}\right)\rightharpoonup y\ {\rm in}\ Y.
    \end{equation*}
    Since $F$ is weakly sequentially closed, we conclude that $F\left(x\right)=y$. This completes the proof of the lemma.
\end{proof}

\par By the stipulations outlined in Lemma \ref{lem2.4}, we can derive the existence, stability, and convergence results about the regularized solution by following the established arguments of classical Tikhonov theory \cite{EHN96, IJ14}. To streamline this discussion, we refrain from including the associated proofs in this article.

\begin{theorem}\label{the2.6}
    \textbf{(Existence)} For any $y^{\delta}\in Y$ and $\alpha>\beta> 0$, there exists at least one minimizer of $\mathcal{J}_{\alpha,\beta}^{\delta}\left(x\right)$ in $\ell_{2}$.
\end{theorem}

\begin{theorem}\label{the2.7}
    \textbf{(Stability)} Let $\alpha_{n}>\beta_{n}> 0$, $\alpha_{n}\rightarrow \alpha$ $\left(\alpha>0\right)$, $\beta_{n}\rightarrow \beta$ $\left(\beta>0\right)$ as $n\rightarrow\infty$. Let the sequence $\left\{y_{n}\right\}\subset Y$ be convergent to $y^{\delta}\in Y$, and let $x_{n}$ be a minimizer to $\mathcal{J}_{\alpha_{n},\beta_{n}}^{\delta_{n}}\left(x\right)$. Then the sequence $\left\{x_{n}\right\}$ contains a subsequence $\left\{x_{n_{k}}\right\}$ converging to a minimizer of $\mathcal{J}_{\alpha,\beta}^{\delta}\left(x\right)$. Furthermore, if $\mathcal{J}_{\alpha,\beta}^{\delta}(x)$ has an unique minimizer $x_{\alpha,\beta}^{\delta}$, then $\lim_{n\rightarrow +\infty}\left\|x_{n}- x_{\alpha,\beta}^{\delta}\right\|_{\ell_{2}}=0$.
\end{theorem}

\begin{theorem}\label{the2.8}
    \textbf{(Convergence)} Let $\alpha_{n}:=\alpha\left(\delta_{n}\right)$, $\beta_{n}:=\beta\left(\delta_{n}\right)$, $\alpha_{n}>\beta_{n}> 0$ satisfy
    \begin{equation*}
        \lim_{n\rightarrow +\infty}\alpha_{n}=0,\quad \lim_{n\rightarrow +\infty}\beta_{n}=0\quad {\rm and} \quad \lim_{n\rightarrow +\infty}\frac{\delta_{n}^{p}}{\alpha_{n}}=0.
    \end{equation*}
    Assume $\eta=\lim_{n\rightarrow +\infty}\eta_{n}\in\left[0,1\right)$ exists, where $\eta_{n}=\beta_{n}/\alpha_{n}$. Let $\delta_{n}\rightarrow 0$ as $n\rightarrow +\infty$ and $y^{\delta_{n}}$ satisfy $\left\|y-y^{\delta_{n}}\right\|\leq\delta_{n}$. Moreover, let
    \begin{equation*}
        x_{\alpha_{n},\beta_{n}}^{\delta_{n}}\in \arg\min_{x\in\ell_{2}} \mathcal{J}_{\alpha_{n},\beta_{n}}^{\delta_{n}}\left(x\right).
    \end{equation*}
    Then $\left\{x_{\alpha_{n},\beta_{n}}^{\delta_{n}}\right\}$ has a subsequence, denoted by $\left\{x_{\alpha_{n_{k}},\beta_{n_{k}}}^{\delta_{n_{k}}}\right\}$ converging to an $\mathcal{R}_{\eta}$-minimizing solution $x^{\dagger}$ in $\ell_{2}$. Furthermore, if the $\mathcal{R}_{\eta}$-minimizing solution $x^{\dagger}$ is unique, then the entire sequence $\left\{x_{\alpha_{n},\beta_{n}}^{\delta_{n}}\right\}$ converges to $x^{\dagger}$ in $\ell_{2}$.
\end{theorem}

\subsection{The case $\eta=1$}\label{sec2.2}

\par We now consider the case where $\eta=1$. In this scenario, the functional $\mathcal{R}_{\eta}\left(x\right)$ remains weakly lower semi-continuous, as shown in \cite[Lemma 2.7, Remark 2.8]{LD25}. However, the properties of coercivity and the Radon-Riesz theorem cannot be extended to this case, as illustrated in Examples \ref{exa2.9} and \ref{exa2.11} below.

\begin{example}\label{exa2.9}
    \textbf{(Non-coercivity)} Let $x=te_{i}$ for some $i\in\mathbb{N}$, where $e_{i}=(\underbrace{0,\cdots,0,1}_{i},0,\cdots)$. In this case, we observe that $\left\|x\right\|_{\ell_{2}}\rightarrow\infty$ as $t\rightarrow\infty$. However, $\mathcal{R}_{1}\left(x\right)=0$ for all $t$. Therefore, $\mathcal{R}_{1}\left(x\right)$ is not coercive.
\end{example} 

\par It is essential to note that the standard proof for the well-posedness of Tikhonov regularization is not valid without the coercivity of the regularization term. To ensure the well-posedness of the problem \eqref{equ1.2} in the case where $\eta=1$, we present a result that imposes an additional assumption, specifically coercivity, on the nonlinear operator $F$; for examples of nonlinear (or linear) coercive operators, see \cite{AZ00, IK13}.

\begin{lemma}\label{lem2.10}
    Assume $F\left(x\right)$ is coercive with respect to $\left\|x\right\|_{\ell_{2}}$, i.e., $\left\|x\right\|_{\ell_{2}}\rightarrow\infty$ implies $\left\|F\left(x\right) \right\|_{Y}\rightarrow\infty$. Then the functional $\mathcal{J}_{\alpha,1}^{\delta}\left(x\right)$ is coercive.
\end{lemma}

\begin{proof}
    By the definition of $\mathcal{J}_{\alpha,\eta}^{\delta}\left(x\right)$, we have
    \begin{align*}
        \mathcal{J}_{\alpha,1}^{\delta}\left(x\right)
        & =\frac{1}{q}\left\|F\left(x\right)- y^{\delta}\right\|_{Y}^{q}+\alpha\left\|x\right\|_{\ell_{1}}^{2}- \alpha\left\|x\right\|_{\ell_{2}}^{2} \\
        & \geq \frac{1}{q}\left|\left\|F\left(x\right)\right\|_{Y}-\left\|y^{\delta}\right\|_{Y} \right|^{q}.
    \end{align*}
    Since $F$ is coercive, it follows that $\mathcal{J}_{\alpha,1}^{\delta}\rightarrow \infty$ as $\left\|x\right\|_{\ell_{2}}\rightarrow\infty$.
\end{proof}

\par Finally, we provide an example to illustrate that $x_{n}$ does not necessarily converge strongly to $x$ even when $x_{n}\rightharpoonup x$ in $\ell_{2}$ and $\mathcal{R}_{1}\left(x_{n}\right)\rightarrow\mathcal{R}_{1}\left(x\right)$. This example shows that $\mathcal{R}_{1}\left(x\right)$ fails to satisfy the Radon-Riesz property.

\begin{example}\label{exa2.11}
    \textbf{(Non-Radon-Riesz property)} Let $x_{n}=(\underbrace{0,\cdots,0,1}_{n},0,\cdots)$ and $x=0$. It follows that $x_{n}\rightharpoonup x$ in $\ell_{2}$. We have
    \begin{equation*}
        \mathcal{R}_{1}\left(x_{n}\right)=\left\|x_{n}\right\|_{\ell_{1}}^{2}- \left\|x_{n}\right\|_{\ell_{2}}^{2}=0 \quad {\rm and} \quad \mathcal{R}_{1}\left(x\right)=0.
    \end{equation*}
    Hence, $\mathcal{R}_{1}\left(x_{n}\right)\rightarrow \mathcal{R}_{1}\left(x\right)$. However, since $\left\|x_{n}-x\right\|_{\ell_{2}}=1$,  it follows that $x_{n}$ does not converge strongly to $x$.
\end{example}

\par Given that $\mathcal{R}_{1}\left(x\right)$ fails to satisfy the Radon-Riesz property, we cannot guarantee the stability and convergence properties similar to those stated in Theorems \ref{the2.7} and \ref{the2.8}. Nevertheless, by introducing the additional assumption of the coercivity of $F$ and utilizing Lemma \ref{lem2.10}, we can establish stability and convergence properties akin to those detailed in Theorems \ref{the2.7} and \ref{the2.8}.

\subsection{Sparsity}\label{sec2.3}

\par Next, we will discuss the sparsity of the regularized solution. When $q=2$, under a restriction on the nonlinearity of $F$, it can be shown that every minimizer of $\mathcal{J}_{\alpha,\eta}^{\delta}\left(x\right)$ is sparse whenever $0<\eta<1$ or $\eta=1$.

\begin{assumption}\label{ass2.12}
    Assume that $F:\ell_{2}\rightarrow Y$ is Fr\'{e}chet-differentiable, and there exists a constant $\gamma>0$ such that  
    \begin{equation}\label{equ2.16}
        \left\|F'\left(y\right)-F'\left(x\right)\right\|_{L(\ell_{2}, Y)}\leq \gamma\left\|y-x\right\|_{\ell_{2}}
    \end{equation}
    for any $y\in B_{\delta}\left(x\right)$, where $x$ is a minimizer of $\mathcal{J}_{\alpha,\eta}^{\delta}\left(x\right)$ and $B_{\delta}\left(x\right):= \left\{y\mid \left\|y-x\right\|_{\ell_{2}}\leq \delta\right\}$, with $\delta\geq\left\|x \right\|_{\infty}$.
\end{assumption}

\begin{remark}\label{rem2.13}
    According to {\rm \cite[p.14]{JM12}}, \eqref{equ2.16} implies
    \begin{equation}\label{equ2.17}
        \left\|F\left(y\right)-F\left(x\right)-F'\left(x\right)\left(y-x\right) \right\|_{Y}\leq\frac{\gamma}{2}\left\|y-x\right\|_{\ell_{2}}^{2}
    \end{equation}
    for any $y\in B_{\delta}\left(x\right)$, where $\delta\geq\left\|x \right\|_{\infty}$.
\end{remark}

\begin{theorem}\label{the2.14}
    \textbf{(Sparsity)} Let $x$ be a minimizer of $\mathcal{J}_{\alpha,\eta}^{\delta} \left(x\right)$ for $0< \eta\leq 1$ and $q=2$. If Assumption \ref{ass2.12} holds, then $x$ is sparse.
\end{theorem}

\begin{proof}
    For $i\in\mathbb{N}$, consider $\overline{x}:=x-x_{i}e_{i}$, where $x_{i}$ is the $i$th component of $x$. It is clear that $\overline{x}\in B_{\delta}\left(x\right)$. By the definition of $x$,
    \begin{equation}\label{equ2.18}
        \frac{1}{2}\left\|F\left(x\right)-y^{\delta}\right\|_{Y}^{2}+ \alpha\mathcal{R}_{\eta}\left(x\right) \leq \frac{1}{2}\left\|F\left(\overline{x}\right)-y^{\delta}\right\|_{Y}^{2}+ \alpha\mathcal{R}_{\eta}\left(\overline{x}\right).
    \end{equation}
    If $x=0$, then $x$ is sparse. Suppose $x\neq 0$, by \eqref{equ2.18}, we see that
    \begin{align}\label{equ2.19}
        \left|x_{i}\right|\left(\left\|x\right\|_{\ell_{1}}+ \left\|\overline{x}\right\|_{\ell_{1}}\right)-\eta\left|x_{i}\right|^{2}
        & = \mathcal{R}_{\eta}\left(x\right)-\mathcal{R}_{\eta}\left(\overline{x}\right) \nonumber\\
        & \leq \frac{1}{2\alpha}\left\|F\left(\overline{x}\right)-y^{\delta}\right\|_{Y}^{2} -\frac{1}{2\alpha}\left\|F\left(x\right)-y^{\delta}\right\|_{Y}^{2} \nonumber\\
        & = \frac{1}{2\alpha}\left\|F\left(\overline{x}\right)-F\left(x\right)\right\|_{Y}^{2} +\frac{1}{\alpha}\left<F\left(\overline{x}\right)-F\left(x\right), F\left(x\right)-y^{\delta}\right>. 
    \end{align}
    From Remark \ref{rem2.13}, we find that
    \begin{equation}\label{equ2.20}
        F\left(\overline{x}\right)=F\left(x\right)+F'\left(x\right)\left(\overline{x}-x\right) +r_{\alpha}^{\delta}
    \end{equation}
    with
    \begin{equation}\label{equ2.21}
        \left\|r_{\alpha}^{\delta}\right\|_{Y}\leq\frac{\gamma}{2}\left\|\overline{x} -x\right\|_{\ell_{2}}^{2}.
    \end{equation}
    Combining \eqref{equ2.20} and \eqref{equ2.21} implies that 
    \begin{align}\label{equ2.22}
        \left\|F\left(\overline{x}\right)-F\left(x\right)\right\|_{Y}^{2}
        & =  \left\|F'\left(x\right)\left(\overline{x}-x\right)\right\|_{Y}^{2} +\left\|r_{\alpha}^{\delta}\right\|_{Y}^{2}+ 2\left<F'\left(x\right)\left(\overline{x}-x\right), r_{\alpha}^{\delta}\right>   \nonumber\\
        & \leq \left\|F'\left(x\right)\right\|_{L\left(\ell_{2}, Y\right)}^{2}\left\|\overline{x}-x\right\|_{\ell_{2}}^{2}+ \frac{\gamma^{2}}{4}\left\|\overline{x}-x\right\|_{\ell_{2}}^{4}+ \gamma\left\|F'\left(x\right)\right\|_{L\left(\ell_{2}, Y\right)}^{2}\left\|\overline{x}-x\right\|_{\ell_{2}}^{3} \nonumber\\
        & = \left|x_{i}\right|^{2}\left\|F'\left(x\right)\right\|_{L\left(\ell_{2}, Y\right)}^{2}+\frac{\gamma^{2}}{4}\left|x_{i}\right|^{4}+ \gamma\left|x_{i}\right|^{3}\left\|F'\left(x\right)\right\|_{L\left(\ell_{2}, Y\right)}^{2}.
    \end{align}
    Moreover,
    \begin{align}\label{equ2.23}
        \left<F\left(\overline{x}\right)-F\left(x\right), F\left(x\right)-y^{\delta}\right>
        & = \left<F'\left(x\right)\left(\overline{x}-x\right) +r_{\alpha}^{\delta}, F\left(x\right)-y^{\delta}\right> \nonumber\\
        & \leq -x_{i}\left<e_{i}, F'\left(x\right)^{*}\left(F\left(x\right) -y^{\delta}\right)\right>+\frac{\gamma}{2}\left|x_{i}\right|^{2} \left\|F\left(x\right)-y^{\delta}\right\|_{Y}^{2}.
    \end{align}
    A combination of \eqref{equ2.19}, \eqref{equ2.22} and \eqref{equ2.23} yields
    \begin{align}\label{equ2.24}
        \left|x_{i}\right|\left(\left\|x\right\|_{\ell_{1}}+ \left\|\overline{x}\right\|_{\ell_{1}}\right)-\eta\left|x_{i}\right|^{2}
        & \leq \frac{1}{2\alpha}\left|x_{i}\right|^{2}\left\|F'\left(x\right) \right\|_{L\left(\ell_{2},Y\right)}^{2}+\frac{\gamma^{2}}{8\alpha} \left|x_{i}\right|^{4}+\frac{\gamma}{2\alpha}\left|x_{i}\right|^{3} \left\|F'\left(x\right)\right\|_{L\left(\ell_{2}, Y\right)}^{2}  \nonumber\\
        &  \quad -\frac{1}{\alpha}x_{i}\left<e_{i},F'\left(x\right)^{*} \left(F\left(x\right)-y^{\delta}\right)\right>+ \frac{\gamma}{2\alpha} \left|x_{i}\right|^{2}\left\|F\left(x\right)-y^{\delta}\right\|_{Y}^{2}
    \end{align}
    for every $i\in\mathbb{N}$. Define $\left\|x\right\|_{\ell_{0}}:=\sum_{i\in\mathbb{N}}{\rm sgn}\left(\left|x_{i}\right| \right)$, where ${\rm sgn}$ is the sign function. If $\left\|x\right\|_{\ell_{0}}=1$, then $x$ is considered sparse. Otherwise, if $\left\|x\right\|_{\ell_{0}}\geq 2$, we have $\left|x_{i}\right|< \left\|x\right\|_{\ell_{1}}+\left\|\overline{x}\right\|_{\ell_{1}}$. Consequently, there exists a constant $c>0$ such that
    \begin{equation}\label{equ2.25}
        \frac{c+\eta\left|x_{i}\right|}{\left\|x\right\|_{\ell_{1}}+ \left\|\overline{x}\right\|_{\ell_{1}}}\leq 1, \quad i.e., \quad \frac{c}{\left\|x\right\|_{\ell_{1}}+ \left\|\overline{x}\right\|_{\ell_{1}}} \leq 1-\frac{\eta\left|x_{i}\right|}{\left\|x\right\|_{\ell_{1}}+ \left\|\overline{x}\right\|_{\ell_{1}}}.
    \end{equation}
    Multiplying \eqref{equ2.25} by $\left|x_{i}\right|$, we obtain
    \begin{equation}\label{equ2.26}
        c\frac{\left|x_{i}\right|}{\left\|x\right\|_{\ell_{1}}+ \left\|\overline{x}\right\|_{\ell_{1}}} \leq \left|x_{i}\right|-\eta\frac{\left|x_{i}\right|^{2}}{\left\|x\right\|_{\ell_{1}}+ \left\|\overline{x}\right\|_{\ell_{1}}}.
    \end{equation}
    Let us denote
    \begin{align*}
        K_{i}
        & := \frac{\left(\left\|x\right\|_{\ell_{1}}+ \left\|\overline{x}\right\|_{\ell_{1}}\right)\left(\frac{1}{2}x_{i} \left\|F'\left(x\right)\right\|_{L\left(\ell_{2},Y\right)}^{2}+\frac{\gamma^{2}}{8} x_{i}^{3}+\frac{\gamma}{2}x_{i}\left|x_{i}\right| \left\|F'\left(x\right)\right\|_{L\left(\ell_{2},Y\right)}^{2}\right)}{c\alpha} \\
        & \qquad + \frac{\left(\left\|x\right\|_{\ell_{1}}+ \left\|\overline{x}\right\|_{\ell_{1}}\right)\left(-x_{i}\left<e_{i}, F'\left(x\right)^{*} \left(F\left(x\right)-y^{\delta}\right)\right>+ \frac{\gamma}{2}x_{i}\left\|F\left(x\right)- y^{\delta}\right\|_{Y}^{2}\right)}{c\alpha}.
    \end{align*}
    A combination of \eqref{equ2.24} and \eqref{equ2.26} implies that
    \begin{equation*}
        K_{i}x_{i}\geq\left|x_{i}\right|, \quad i\in\mathbb{N}.
    \end{equation*}
    Since $x$ is a minimizer of $\mathcal{J}_{\alpha,\eta}^{\delta}\left(x\right)$, $\left\|F\left(x\right)-y^{\delta}\right\|_{Y}$ is finite. Additionally, since $F$ is Fr\'{e}chet-differentiable, $\left\|F'\left(x\right)\right\|_{L\left(\ell_{2}, Y\right)}$ is also finite. Notably, since $x,\overline{x}\in\ell_{2}$, $F'\left(x\right)^{*} \left(F\left(x\right)-y^{\delta}\right)\in\ell_{2}$, it follows that $K_{i}\rightarrow 0$ as $i\rightarrow\infty$. This implies that $\Lambda:=\left\{i\in\mathbb{N}\mid\left| K_{i}\right|\geq 1\right\}$ is finite. It is evident that $x_{i}=0$ whenever $i\notin\Lambda$, thereby proving the proposition.
\end{proof}

\par Note that the above result holds only for the case when $q=2$. It is unclear whether each minimizer of $\mathcal{J}_{\alpha,\eta}^{\delta}\left(x\right)$ is sparse when $q\neq 2$.

\section{Convergence rate of the regularized solutions}\label{sec3}

\par We analyze the convergence rate for the scenario where $0<\eta<1$ in this section. To this end, we impose specific smoothness conditions on the solution $x^{\dagger}$. Additionally, we enact two commonly accepted assumptions regarding the nonlinearity of $F$, and subsequently derive two corresponding inequalities. From these inequalities, we ascertain convergence rates of $\mathcal{O}\left(\delta^{1/2}\right)$ and $\mathcal{O}\left(\delta\right)$ in the $\ell_{2}$-norm.

\subsection{Convergence rate $\mathcal{O}\left(\delta^{1/2}\right)$}\label{sec3.1}

\begin{assumption}\label{ass3.1}
    Let $x^{\dagger}\neq 0$ be an $\mathcal{R}_{\eta}$-minimizing solution to the problem \eqref{equ1.1}, which exhibits sparsity. Assume that \\
    (1) $F$ is continuously Fr\'{e}chet-differentiable. For each $i\in I\left(x^{\dagger}\right)$, there exists $\omega_{i}\in Y$ such that
    \begin{equation}\label{equ3.1}
        e_{i} = F'\left(x^{\dagger}\right)^{*}\omega_{i},
    \end{equation}
    where $I\left(x^{\dagger}\right)$ is defined in \eqref{equ2.2}. \\
    (2) There exists constants $\gamma>0$, $\rho>0$, and $\mathcal{R}_{\eta}\left(x^{\dagger}\right)<\rho$ such that
    \begin{equation}\label{equ3.2}
        \left\|F'\left(x\right)-F'\left(x^{\dagger}\right)\right\|_{L\left(\ell_{2}, Y\right)}\leq \gamma\left\|x-x^{\dagger}\right\|_{\ell_{2}}
    \end{equation}
    for all $x\in\ell_{2}$ that satisfy $\mathcal{R}_{\eta}\left(x\right)\leq \rho$.
\end{assumption}

\par Assumption \ref{ass3.1} (1) and similar conditions have been established in previous works \cite{BL09, G09}. Specifically, Assumption \ref{ass3.1} (1) serves as a source condition, imposing a smoothness requirement on the solution $x^{\dagger}$. In contrast, Assumption \ref{ass3.1} (2) imposes a restriction on $F$, which has significant implications. It not only places a nonlinearity condition on $F$ but also plays a crucial role in estimating the term $\left<F'\left(x^{\dagger}\left(x-x^{\dagger}\right)\right),\omega_{i}\right>$, where $\omega_{i}$ are defined as in \eqref{equ3.1}. Many researchers have noted that combining nonlinearity restrictions on $F$ with source conditions is a powerful framework for establishing convergence rates in regularization contexts \cite{GHS08, HKPS07, SKHK12}. One commonly utilized restriction is \eqref{equ3.2}, which asserts that $F'$ is Lipschitz continuous \cite{EHN96, JM12}

\begin{remark}\label{rem3.2}
    Note that for all $x\in\ell_{2}$ satisfying $\mathcal{R}_{\eta}\left(x\right)\leq \rho$, \eqref{equ3.2} implies
    \begin{equation}\label{equ3.3}
        \left\|F\left(x\right)-F\left(x^{\dagger}\right)-F'\left(x^{\dagger}\right) \left(x-x^{\dagger}\right)\right\|_{Y}\leq \frac{\gamma}{2}\left\|x-x^{\dagger} \right\|_{\ell_{2}}^{2}
    \end{equation}
    according to a Taylor approximation of $F$. By applying the triangle inequality, we deduce
    \begin{equation}\label{equ3.4}
        \left\|F'\left(x^{\dagger}\right)\left(x-x^{\dagger}\right)\right\|_{Y}\leq \frac{\gamma}{2}\left\|x-x^{\dagger}\right\|_{\ell_{2}}^{2}+ \left\|F\left(x\right)-F\left(x^{\dagger}\right)\right\|_{Y}
    \end{equation}
    for all $x\in\ell_{2}$ satisfying $\mathcal{R}_{\eta}\left(x\right)\leq\rho$. This provides an upper bound for the term $\left<F'\left(x^{\dagger}\right) \left(x-x^{\dagger}\right),\omega_{i}\right>$. 
\end{remark}

\begin{lemma}\label{lem3.3}
    If Assumption \eqref{equ3.1} holds, it follows that there exists $x^{\dagger}= F'\left(x^{\dagger}\right)^{*}\omega^{\dagger}$. 
\end{lemma}

\par This assertion can be readily verified by choosing $\omega^{\dagger}=\sum_{i\in I\left(x^{\dagger}\right)}x_{i}^{\dagger}\omega_{i}$.

\par Subsequently, we derive an inequality necessary for the proof of the convergence rate. According to \ref{lem2.4} (1), for any $M>0$, there exists $M_{1}>0$ such that $\mathcal{R}_{\eta}\left(x\right)\leq M$ for $x\in\ell_{2}$ implies $\left\|x\right\|_{\ell_{1}}\leq M_{1}$. We further denote
\begin{align}
    c_{1} &= \left(M_{1}+2\left\|x^{\dagger}\right\|_{\ell_{1}}\right) \left|I\left(x^{\dagger}\right)\right|\max_{i\in I\left(x^{\dagger}\right)} \left\|\omega_{i}\right\|_{Y}+2\left\|\omega^{\dagger}\right\|_{Y},  \label{equ3.5} \\
    c_{2} &= 2\left\|\omega^{\dagger}\right\|_{Y}.  \label{equ3.6}
\end{align}

\begin{lemma}\label{lem3.4}
    Let $M>0$ be given, and denote by $c_{1}$ and $c_{2}$ the constants defined in \eqref{equ3.5}-\eqref{equ3.6}. Under Assumption \ref{ass3.1}, if $\Gamma:=\frac{\gamma \left(c_{1}-c_{2}\eta\right)}{2\left(1-\eta\right)}<1$, then we have
    \begin{equation*}
        \left\|x-x^{\dagger}\right\|_{\ell_{2}}^{2}\leq \frac{1}{1-\Gamma} \left[\frac{1}{1-\eta}\left(\mathcal{R}_{\eta}\left(x\right)- \mathcal{R}_{\eta}\left(x^{\dagger}\right)\right)-\frac{c_{1}-c_{2}\eta} {1-\eta}\left\|F\left(x\right)- F\left(x^{\dagger}\right)\right\|_{Y}\right]
    \end{equation*}
    for any $x\in\ell_{2}$ satisfying $\mathcal{R}_{\eta}\left(x\right)\leq M$ and $\mathcal{R}_{\eta}\left(x\right)\leq \rho$.
\end{lemma}

\begin{proof}
    By the definition of $\mathcal{R}_{\eta}\left(x\right)$ in \eqref{equ1.3}, we can expound that 
    \begin{align*}
        \mathcal{R}_{\eta}\left(x\right)
        & = \left\|x\right\|_{\ell_{1}}^{2}-\eta\left\|x\right\|_{\ell_{2}}^{2} = \left(\left\|x\right\|_{\ell_{1}}^{2}-\left\|x\right\|_{\ell_{2}}^{2}\right) +\left(1-\eta\right)\left\|x\right\|_{\ell_{2}}^{2} \\
        & = \mathcal{K}\left(x\right)+\left(1-\eta\right)\left\|x\right\|_{\ell_{2}}^{2},
    \end{align*}
    where $\mathcal{K}\left(x\right):=\left\|x\right\|_{\ell_{1}}^{2}- \left\|x\right\|_{\ell_{2}}^{2}$. Thus, we observe that 
    \begin{align}\label{equ3.7}
        \mathcal{R}_{\eta}\left(x\right)-\mathcal{R}_{\eta}\left(x^{\dagger}\right)
        & = \mathcal{K}\left(x\right)-\mathcal{K}\left(x^{\dagger}\right) +\left(1-\eta\right)\left(\left\|x\right\|_{\ell_{2}}^{2}- \left\|x^{\dagger}\right\|_{\ell_{2}}^{2}\right) \nonumber\\
        & = \mathcal{K}\left(x\right)-\mathcal{K}\left(x^{\dagger}\right) +\left(1-\eta\right)\left\|x-x^{\dagger}\right\|_{\ell_{2}}^{2} +2\left(1-\eta\right)\left<x^{\dagger}, x-x^{\dagger}\right>.  
    \end{align}
    From the definition of $\mathcal{K}\left(x\right)$, we can express
    \begin{equation*}
        \mathcal{K}\left(x\right)-\mathcal{K}\left(x^{\dagger}\right) = \left\|x\right\|_{\ell_{1}}^{2}-\left\|x\right\|_{\ell_{2}}^{2} - \left\|x^{\dagger}\right\|_{\ell_{1}}^{2}+ \left\|x^{\dagger}\right\|_{\ell_{2}}^{2}.
    \end{equation*}
    With the definition of the index set $I\left(x^{\dagger}\right)$ in \eqref{equ2.2}, we can derive that
    \begin{align*}
      \left\|x\right\|_{\ell_{1}}^{2} & = \left(\sum_{i\in I\left(x^{\dagger}\right)} \left|x_{i}\right| + \sum_{i\notin I\left(x^{\dagger}\right)}\left|x_{i}\right| \right)^{2} \geq \left(\sum_{i\in I\left(x^{\dagger}\right)} \left|x_{i}\right| \right)^{2} + \sum_{i\notin I\left(x^{\dagger}\right)}\left|x_{i}\right|^{2}, \\
      -\left\|x\right\|_{\ell_{2}}^{2} & = -\sum_{i\in I\left(x^{\dagger}\right)} \left|x_{i}\right|^{2} - \sum_{i\notin I\left(x^{\dagger}\right)} \left|x_{i}\right|^{2}.
    \end{align*}
    Consequently, we obtain
    \begin{equation*}
        \mathcal{K}\left(x\right)-\mathcal{K}\left(x^{\dagger}\right) \geq \left(\sum_{i\in I\left(x^{\dagger}\right)}\left|x_{i}\right|\right)^{2} -\left(\sum_{i\in I\left(x^{\dagger}\right)} \left|x_{i}^{\dagger}\right|\right)^{2} - \sum_{i\in I\left(x^{\dagger}\right)} \left(\left|x_{i}\right|^{2} - \left|x_{i}^{\dagger}\right|^{2}\right),
    \end{equation*}
    which can be rewritten as
    \begin{align}\label{equ3.8}
        \mathcal{K}\left(x\right)-\mathcal{K}\left(x^{\dagger}\right) 
        & \geq \left(\sum_{i\in I\left(x^{\dagger}\right)}\left|x_{i}\right| + \sum_{i\in I\left(x^{\dagger}\right)} \left|x_{i}^{\dagger}\right|\right)\left(\sum_{i\in I\left(x^{\dagger}\right)}\left(\left|x_{i}\right|-\left|x_{i}^{\dagger}\right |\right)\right) \nonumber\\
        & \quad -\sum_{i\in I\left(x^{\dagger}\right)} \left(\left|x_{i}\right|+\left|x_{i}^{\dagger}\right|\right) \left(\left|x_{i}\right|-\left|x_{i}^{\dagger}\right|\right).
    \end{align}
    By the definition of $M_{1}$, we find
    \begin{equation*}
        \sum_{i\in I\left(x^{\dagger}\right)}\left|x_{i}\right| + \sum_{i\in I\left(x^{\dagger}\right)} \left|x_{i}^{\dagger}\right| \leq M_{1} + \left\|x^{\dagger}\right\|_{\ell_{1}}.
    \end{equation*}
    Thus, we conclude that
    \begin{equation}\label{equ3.9}
        \left|x_{i}\right| + \left|x_{i}^{\dagger}\right| \leq M_{1} + \left\|x^{\dagger}\right\|_{\ell_{1}}, \quad \forall i\in I\left(x^{\dagger}\right).
    \end{equation}
    Meanwhile, we observe that
    \begin{equation}\label{equ3.10}
        0<\left\|x^{\dagger}\right\|_{\ell_{1}}\leq \sum_{i\in I\left(x^{\dagger}\right)}\left|x_{i}\right| + \sum_{i\in I\left(x^{\dagger}\right)} \left|x_{i}^{\dagger}\right|.
    \end{equation}
    By combining inequalities \eqref{equ3.8}, \eqref{equ3.9} and \eqref{equ3.10}, we derive that
    \begin{equation*}
        \mathcal{K}\left(x\right)-\mathcal{K}\left(x^{\dagger}\right) \geq -\left\|x^{\dagger}\right\|_{\ell_{1}} \sum_{i\in I\left(x^{\dagger}\right)}\left|x_{i}-x_{i}^{\dagger}\right| - \left(M_{1}+\left\|x^{\dagger}\right\|_{\ell_{1}}\right) \sum_{i\in I\left(x^{\dagger}\right)}\left|x_{i}-x_{i}^{\dagger}\right|,
    \end{equation*}
    i.e.,
    \begin{equation}\label{equ3.11}
        \mathcal{K}\left(x\right)-\mathcal{K}\left(x^{\dagger}\right) \geq -\left(M_{1}+2\left\|x^{\dagger}\right\|_{\ell_{1}}\right) \sum_{i\in I\left(x^{\dagger}\right)}\left|x_{i}-x_{i}^{\dagger}\right|.
    \end{equation}
    A combination of \eqref{equ3.7} and \eqref{equ3.11} leads to the following inequality
    \begin{align}\label{equ3.12}
        \left(1-\eta\right)\left\|x-x^{\dagger}\right\|_{\ell_{2}}^{2}
        & \leq \mathcal{R}_{\eta}\left(x\right)-\mathcal{R}_{\eta}\left(x^{\dagger}\right) + \left(M_{1}+2\left\|x^{\dagger}\right\|_{\ell_{1}}\right) \sum_{i\in I\left(x^{\dagger}\right)}\left|x_{i}-x_{i}^{\dagger}\right| \nonumber\\
        & \quad -2\left(1-\eta\right)\left<x^{\dagger}, x-x^{\dagger}\right>.
    \end{align}
    By referring to Assumption \ref{ass3.1} (1), we find that
    \begin{equation*}
        \left|x_{i}-x^{\dagger}_{i}\right|=\left|\left<e_{i}, x-x^{\dagger}\right> \right| =\left|\left<\omega_{i}, F'\left(x^{\dagger}\right) \left(x-x^{\dagger}\right)\right>\right|\leq \max_{i\in I\left(x^{\dagger}\right)}\left\|\omega_{i}\right\|_{Y} \left\|F'\left(x^{\dagger}\right)\left(x-x^{\dagger}\right)\right\|_{Y}.
    \end{equation*}
    Thus, we can conclude that
    \begin{equation}\label{equ3.13}
        \sum_{i\in I\left(x^{\dagger}\right)}\left|x_{i}-x_{i}^{\dagger}\right| \leq \left|I\left(x^{\dagger}\right)\right|\max_{i\in I\left(x^{\dagger}\right)} \left\|\omega_{i}\right\|_{Y} \left\|F'\left(x^{\dagger}\right)\left(x-x^{\dagger}\right)\right\|_{Y},
    \end{equation} 
    where $\left|I\left(x^{\dagger}\right)\right|$ denotes the cardinality of the index set $I\left(x^{\dagger}\right)$. Additionally, according to Lemma \ref{lem3.3}, we can ascertain that
    \begin{equation}\label{equ3.14}
        \left|\left<x^{\dagger}, x-x^{\dagger}\right>\right| = \left|\left<\omega^{\dagger}, F'\left(x\right)\left(x-x^{\dagger}\right) \right>\right| \leq \left\|\omega^{\dagger}\right\|_{Y} \left\|F'\left(x^{\dagger}\right)\left(x-x^{\dagger}\right)\right\|_{Y}.
    \end{equation}
    A combination of \eqref{equ3.4}, \eqref{equ3.12}, \eqref{equ3.13} and \eqref{equ3.14} implies that
    \begin{align*}
        & \quad \left(1-\eta\right)\left\|x-x^{\dagger}\right\|_{\ell_{2}}^{2} \\
        & \leq \mathcal{R}_{\eta}\left(x\right)-\mathcal{R}_{\eta}\left(x^{\dagger}\right) + \left(\left(M_{1}+2\left\|x^{\dagger}\right\|_{\ell_{1}}\right) \left|I\left(x^{\dagger}\right)\right|\max_{i\in I\left(x^{\dagger}\right)}\left\|\omega_{i}\right\|_{Y} + 2\left(1-\eta\right)\left\|\omega^{\dagger}\right\|_{Y}\right) \\
        & \quad \cdot\left(\frac{\gamma}{2}\left\|x-x^{\dagger}\right\|_{\ell_{2}}^{2}+ \left\|F\left(x\right)-F\left(x^{\dagger}\right)\right\|_{Y}\right),
    \end{align*}
    i.e.,
    \begin{equation*}
        \left\|x-x^{\dagger}\right\|_{\ell_{2}}^{2} \leq \frac{1}{1-\frac{\gamma\left(c_{1}-c_{2}\eta\right)}{2\left(1-\eta\right)}} \left[\frac{1}{1-\eta}\left(\mathcal{R}_{\eta}\left(x\right)- \mathcal{R}_{\eta}\left(x^{\dagger}\right)\right)+ \frac{c_{1}-c_{2}\eta}{1-\eta} \left\|F\left(x\right)-F\left(x^{\dagger}\right)\right\|_{Y}\right],
    \end{equation*}
    where $c_{1}$ and $c_{2}$ are defined by \eqref{equ3.5} and \eqref{equ3.6}.
\end{proof}

\begin{theorem}\label{the3.5}
    Suppose that Assumption \ref{ass3.1} holds. Let $x_{\alpha,\eta}^{\delta}$ be defined by \eqref{equ2.1}, and let the constants $c_{1}>c_{2}>0$ be as specified in Lemma \ref{lem3.4}. Assume $\Gamma:=\frac{\gamma\left(c_{1}-c_{2}\eta\right)}{2\left(1-\eta\right)}<1$. \\
    (1) If $q=1$ and $\alpha\left(c_{1}-c_{2}\eta\right) < 1$, then
    \begin{equation}\label{equ3.15}
        \left\|x_{\alpha,\eta}^{\delta}-x^{\dagger}\right\|_{\ell_{2}}^{2}\leq \frac{\left(1+\alpha\left(c_{1}-c_{2}\eta\right)\right)\delta}{\alpha\left(1-\eta\right) \left(1-\Gamma\right)} ,\quad \left\|F\left(x_{\alpha,\eta}^{\delta}\right)-y^{\delta}\right\|_{Y}\leq \frac{\left(1+\alpha\left(c_{1}-c_{2}\eta\right)\right)\delta}{1- \alpha\left(c_{1}-c_{2}\eta\right)}
    \end{equation}
    (2) If $q>1$, then
    \begin{align}\label{equ3.16}
        & \left\|x_{\alpha,\eta}^{\delta}-x^{\dagger}\right\|_{\ell_{2}}^{2} \leq \frac{1}{\alpha\left(1-\eta\right) \left(1-\Gamma\right)} \left[\frac{\delta^{q}}{q}+\alpha\left(c_{1}-c_{2}\eta\right)\delta+ \frac{\left(q-1\right)2^{\frac{1}{q-1}}\alpha^{\frac{q}{q-1}} \left(c_{1}-c_{2}\eta\right)^{\frac{q}{q-1}}}{q}\right], \nonumber\\ 
        & \left\|F\left(x_{\alpha,\eta}^{\delta}\right)-y^{\delta}\right\|_{Y}^{q} \leq q  \left[\frac{\delta^{q}}{q}+\alpha\left(c_{1}-c_{2}\eta\right)\delta+ \frac{\left(q-1\right)2^{\frac{1}{q-1}}\alpha^{\frac{q}{q-1}} \left(c_{1}-c_{2}\eta\right)^{\frac{q}{q-1}}}{q}\right].
    \end{align}
\end{theorem}

\begin{proof}
    By the definition of $x_{\alpha,\eta}^{\delta}$, it follows that
    \begin{equation*}
        \frac{1}{q}\left\||F\left(x_{\alpha,\eta}^{\delta}\right)-y^{\delta}\right\|_{Y}^{q} + \alpha\mathcal{R}_{\eta}\left(x_{\alpha,\eta}^{\delta}\right) \leq \frac{1}{q}\left\||F\left(x^{\dagger}\right)-y^{\delta}\right\|_{Y}^{q} + \alpha\mathcal{R}_{\eta}\left(x^{\dagger}\right),
    \end{equation*}
    which implies
    \begin{equation*}
        \frac{\delta^{q}}{q} \geq \alpha\mathcal{R}_{\eta}\left(x_{\alpha,\eta}^{\delta}\right) - \alpha\mathcal{R}_{\eta}\left(x^{\dagger}\right) + \frac{1}{q} \left\||F\left(x_{\alpha,\eta}^{\delta}\right)-y^{\delta}\right\|_{Y}^{q}.
    \end{equation*}
    Consequently, since $\mathcal{R}_{\eta}$ is bounded, applying Lemma \ref{lem3.4} yields
    \begin{align}\label{equ3.17}
        \frac{\delta^{q}}{q}
        & \geq \alpha\left(1-\Gamma\right)\left(1-\eta\right)\left\|x_{\alpha,\eta}^{\delta}- x^{\dagger}\right\|_{\ell_{2}}^{2}-\alpha\left(c_{1}-c_{2}\eta\right) \left\|F\left(x_{\alpha,\eta}^{\delta}\right)-F\left(x^{\dagger}\right)\right\|_{Y} +\frac{1}{q}\left\||F\left(x_{\alpha,\eta}^{\delta}\right)-y^{\delta}\right\|_{Y}^{q} \nonumber\\
        & \geq \alpha\left(1-\Gamma\right)\left(1-\eta\right)\left\|x_{\alpha,\eta}^{\delta}- x^{\dagger}\right\|_{\ell_{2}}^{2}-\alpha\left(c_{1}-c_{2}\eta\right) \left\|F\left(x_{\alpha,\eta}^{\delta}\right)-y^{\delta}\right\|_{Y} \nonumber\\ 
        & \quad -\alpha\left(c_{1}-c_{2}\eta\right)\delta +\frac{1}{q} \left\||F\left(x_{\alpha,\eta}^{\delta}\right)-y^{\delta}\right\|_{Y}^{q}.
    \end{align}
    Therefore, if $q=1$ and $\alpha\left(c_{1}-c_{2}\eta\right)< 1$, then \eqref{equ3.15} holds.
    
    \par Conversely, when $q>1$, we apply Young's inequality $ab\leq\frac{a^{q}}{q}+ \frac{b^{q^{*}}}{q^{*}}$ for $a,b>0$ and $q^{*}>1$ satisfying $\frac{1}{q}+\frac{1}{q^{*}}=1$. We obtain
    \begin{align}\label{equ3.18}
        \alpha\left(c_{1}-c_{2}\eta\right) \left\|F\left(x_{\alpha,\eta}^{\delta}\right)-y^{\delta}\right\|_{Y} 
        & = 2^{\frac{1}{q}} \alpha\left(c_{1}-c_{2}\eta\right) 2^{-\frac{1}{q}} \left\|F\left(x_{\alpha,\eta}^{\delta}\right)-y^{\delta}\right\|_{Y} \nonumber\\
        & \leq \frac{1}{2q}\left\|F\left(x_{\alpha,\eta}^{\delta}\right)- y^{\delta}\right\|_{Y} + \frac{\left(q-1\right)2^{\frac{1}{q-1}} \left(c_{1}-c_{2}\eta\right)\frac{q}{q-1}}{q}.
    \end{align}
    A combination of \eqref{equ3.17} and \eqref{equ3.18} leads to the conclusion that \eqref{equ3.16} holds.    
\end{proof}

\par Note that by the definition of $x_{\alpha,\eta}^{\delta}$,
\begin{equation*}
    \frac{1}{q}\left\|F\left(x_{\alpha,\eta}^{\delta}\right)-y^{\delta}\right\|_{Y}^{q}+ \alpha\mathcal{R}_{\eta}\left(x_{\alpha,\eta}^{\delta}\right)\leq \frac{1}{q}\left\|F\left(0\right)-y^{\delta}\right\|_{Y}^{q}+ \alpha\mathcal{R}_{\eta}\left(0\right) = \frac{1}{q}\left\|F\left(0\right)-y^{\delta}\right\|_{Y}^{q}.
\end{equation*}
Hence, the upper bound of $\mathcal{R}_{\eta}\left(x_{\alpha,\eta}^{\delta}\right)$ depends on the noise level $\delta$.

\begin{remark}\label{rem3.6}
    \textbf{(A priori estimation)} Let $\eta\alpha$ be a fixed constant. For the case when $q>1$, if $\alpha\sim\delta^{q-1}$, then $\left\|x_{\alpha,\eta}^{\delta}-x^{\dagger}\right\|_{\ell_{2}}\leq c\delta^{1/2}$ for some constant $c>0$. In the particular case when $q=1$, if $\alpha\sim\delta^{1-\epsilon}$ with $\left(0<\epsilon<1\right)$, then $\left\|x_{\alpha,\eta}^{\delta}-x^{\dagger}\right\|_{\ell_{2}}\leq c\delta^{\epsilon/2}$ for some constant $c>0$.
\end{remark}

\par Note that due to the presence of the term $\frac{\gamma}{2}\left\|x-x^{\dagger}\right\|_{\ell_{2}}^{2}$ in the estimation \eqref{equ3.4}, we need an additional condition to obtain the convergence rate. Specifically, $\gamma>0$ must be small enough for $\Gamma<1$. This additional condition is similar to the requirement $\gamma\left\|\omega\right\|<1$ in classical quadratic regularization \cite{EHN96}.

\begin{theorem}\label{the3.7}
    \textbf{(Discrepancy principle)} Keep the assumptions of Lemma \ref{lem3.4} and let $x_{\alpha,\eta}^{\delta}$ be defined by \eqref{equ2.1}. Assume there exist parameters $\alpha>0$ and $0<\eta<1$ that are determined by the discrepancy principle
    \begin{equation*}
        \delta\leq\left\|F\left(x_{\alpha,\eta}^{\delta}\right)-y^{\delta}\right\|_{Y} \leq \tau\delta,\quad \tau\geq 1.
    \end{equation*}
    Then
    \begin{equation*}
        \left\|x_{\alpha,\eta}^{\delta}-x^{\dagger}\right\|_{\ell_{2}}\leq \sqrt{\frac{\left(c_{1}-c_{2}\eta\right)\left(\tau+1\right)\delta} {\left(1-\Gamma\right)\left(1-\eta\right)}}.
    \end{equation*}
\end{theorem}

\begin{proof}
    By the definition of $x_{\alpha,\eta}^{\delta}$, $\alpha$ and $\eta$, we see that
    \begin{align}\label{equ3.19}
        \frac{\delta^{q}}{q}+\alpha\mathcal{R}_{\eta}\left(x_{\alpha,\eta}^{\delta}\right) & \leq \frac{1}{q}\left\||F\left(x_{\alpha,\eta}^{\delta}\right)- y^{\delta}\right\|_{Y}^{q} + \alpha\mathcal{R}_{\eta} \left(x_{\alpha,\eta}^{\delta}\right) \nonumber\\
        & \leq \frac{1}{q}\left\||F\left(x^{\dagger}\right)-y^{\delta}\right\|_{Y}^{q} + \alpha\mathcal{R}_{\eta}\left(x^{\dagger}\right)
    \end{align}
    Hence, we have $\mathcal{R}_{\eta}\left(x_{\alpha,\eta}^{\delta}\right)\leq \mathcal{R}_{\eta}\left(x^{\dagger}\right)$. From Lemma \ref{lem3.4}, it follows that 
    \begin{align}\label{equ3.20}
        0 \geq \mathcal{R}_{\eta}\left(x_{\alpha,\eta}^{\delta}\right) - \mathcal{R}_{\eta}\left(x^{\dagger}\right)
        & \geq \left(1-\Gamma\right)\left(1-\eta\right) \left\|x_{\alpha,\eta}^{\delta}-x^{\dagger}\right\|_{\ell_{2}}^{2}- \left(c_{1}-c_{2}\eta\right) \left\|F\left(x\right)-F\left(x^{\dagger}\right)\right\|_{Y} \nonumber\\
        & \geq \left(1-\Gamma\right)\left(1-\eta\right) \left\|x_{\alpha,\eta}^{\delta}-x^{\dagger}\right\|_{\ell_{2}}^{2}- \left(c_{1}-c_{2}\eta\right)\left(\tau+1\right)\delta.
    \end{align}
    Then, we obtain
    \begin{equation*}
        \left\|x_{\alpha,\eta}^{\delta}-x^{\dagger}\right\|_{\ell_{2}}^{2}\leq \frac{\left(c_{1}-c_{2}\eta\right)\left(\tau+1\right)\delta} {\left(1-\Gamma\right)\left(1-\eta\right)}.
    \end{equation*}
    This concludes the proof of the theorem.
\end{proof}

\par It is important to note that a drawback of the discrepancy principle is that a regularization parameter satisfying
\begin{equation*}
    \delta\leq\left\|F\left(x_{\alpha,\eta}^{\delta}\right)-y^{\delta}\right\|_{Y} \leq \tau\delta,\quad \tau\geq 1.
\end{equation*}
might not exist for general nonlinear operators $F$. In fact, for nonlinear operators, it is often challenging to ensure the existence of the regularization parameter $\alpha$ as determined by Morozov's discrepancy principle. Therefore, we need to impose specific conditions on $F$. This paper primarily focuses on the convergence rate under Morozov's discrepancy principle. 

\subsection{Convergence rate $\mathcal{O}\left(\delta\right)$}\label{sec3.2}

\par In \cite[p.6]{KNS08}, it is noted that for ill-posed problems, \eqref{equ3.3} does not provide sufficient information about the local behavior of $F$ around $x^{\dagger}$ to draw conclusions regarding convergence. This is because the left-hand side of \eqref{equ3.3} can be significantly smaller than the right-hand side for certain pairs of points $x$ and $x^{\dagger}$. Consequently, several researchers have proposed an alternative condition given by
\begin{equation}\label{equ3.21}
    \left\|F\left(x\right)-F\left(x^{\dagger}\right)-F'\left(x^{\dagger}\right)\left(x- x^{\dagger}\right)\right\|_{Y}\leq \gamma\left\|F\left(x\right)- F\left(x^{\dagger}\right)\right\|_{Y},\quad 0<\gamma<\frac{1}{2}
\end{equation}
as a criterion for the nonlinearity of $F$; see \cite[pp.278-279]{EHN96}, \cite[p.6]{KNS08}, and \cite[PP.69-70]{SKHK12}.

\begin{assumption}\label{ass3.8}
    Let $x^{\dagger}\neq 0$ be an $\mathcal{R}_{\eta}$-minimizing solution of the problem $F\left(x\right)=y$ that is sparse. Furthermore, we assume that \\
    (1) $F$ is Fr\'{e}chet-differentiable at $x^{\dagger}$. For each $i\in I\left(x^{\dagger}\right)$, there exists an $\omega_{i}\in Y$ such that
    \begin{equation}\label{equ3.22}
        e_{i}=F'\left(x^{\dagger}\right)^{*}\omega_{i}
    \end{equation}
    holds, where $I\left(x^{\dagger}\right)$ is defined in \eqref{equ2.2}. \\
    (2) There exists constants $0<\gamma<\frac{1}{2}$, $\rho>0$, and $\mathcal{R}_{\eta}\left(x^{\dagger}\right)<\rho$ such that
    \begin{equation}\label{equ3.23}
        \left\|F\left(x\right)-F\left(x^{\dagger}\right)-F'\left(x^{\dagger}\right) \left(x-x^{\dagger}\right)\right\|_{Y} \leq \gamma\left\|F\left(x\right)-F\left(x^{\dagger}\right)\right\|_{Y}
    \end{equation}
    for any $x\in\ell_{2}$ with $\mathcal{R}_{\eta}\left(x\right)\leq\rho$.
\end{assumption}

\begin{remark}\label{rem3.9}
    Utilizing the triangle inequality, it follows from \eqref{equ3.23} that
    \begin{equation}\label{equ3.24}
        \frac{1}{1+\gamma}\left\|F'\left(x^{\dagger}\right) \left(x-x^{\dagger}\right)\right\|_{Y}\leq \left\|F\left(x\right)-F\left(x^{\dagger}\right)\right\|_{Y}\leq \frac{1}{1-\gamma}\left\|F'\left(x^{\dagger}\right) \left(x-x^{\dagger}\right)\right\|_{Y},
    \end{equation}
    which provides two-sides bounds on $\left\|F'\left(x^{\dagger}\right) \left(x-x^{\dagger}\right)\right\|_{Y}$. The condition
    \begin{equation*}
        \left\|F'\left(x^{\dagger}\right) \left(x-x^{\dagger}\right)\right\|_{Y}\leq \left(1+\gamma\right)\left\|F\left(x\right)-F\left(x^{\dagger}\right)\right\|_{Y}
    \end{equation*}
    has been adopted by several researchers. This assumption directly leads to a bound on the critical inner product $\left<F'\left(x^{\dagger}\right) \left(x-x^{\dagger}\right),\omega_{i}\right>$.
\end{remark}

\par Next, we derive an inequality based on the condition \eqref{equ3.23}. The linear convergence rate $\mathcal{O}\left(\delta\right)$ follows directly from the inequality.

\begin{lemma}\label{lem3.10}
    Let Assumption \ref{ass3.8} hold, and suppose that $\mathcal{R}_{\eta}\left(x\right)\leq M$ for some $M>0$. Then there exist constants $c_{3}>0$ and $c_{4}>c_{5}$ such that
    \begin{equation}\label{equ3.25}
        c_{3}\left(1-\eta\right)\left\|x-x^{\dagger}\right\|_{\ell_{1}}\leq \mathcal{R}_{\eta}\left(x\right)-\mathcal{R}_{\eta}\left(x^{\dagger}\right)+ \left(c_{4}-c_{5}\eta\right)\left\|F\left(x\right)- F\left(x^{\dagger}\right)\right\|_{Y}
    \end{equation}
    for any $x\in\ell_{2}$ satisfying $\mathcal{R}_{\eta}\left(x\right)\leq\rho$.
\end{lemma}

\begin{proof}
    By Lemma \ref{lem2.4} (1), if $M>0$, there exists $M_{1}>0$ such that $\mathcal{R}_{\eta}\left(x\right)\leq M$ for $x\in\ell_{2}$ implies $\left\|x\right\|_{\ell_{1}}\leq M_{1}$. Following a similar reasoning as in the proof of \cite[Lemma 3.7]{LD25}, we obtain
    \begin{equation}\label{equ3.26}
        \left(1-\eta\right)\left\|x-x^{\dagger}\right\|_{\ell_{1}}\leq \frac{\mathcal{R}_{\eta}\left(x\right)-\mathcal{R}_{\eta}\left(x^{\dagger}\right)} {\left\|x^{\dagger}\right\|_{\ell_{1}}} +\left(2+\frac{M_{1}}{\left\|x^{\dagger}\right\|_{\ell_{1}}}\eta\right) \sum_{i\in I\left(x^{\dagger}\right)}\left|x_{i}-x_{i}^{\dagger}\right|.
    \end{equation} 
    Additionally, by Assumption \ref{ass3.8},
    \begin{equation*}
        \left|x_{i}-x^{\dagger}_{i}\right|=\left|\left<e_{i}, x-x^{\dagger}\right> \right| =\left|\left<\omega_{i}, F'\left(x^{\dagger}\right) \left(x-x^{\dagger}\right)\right>\right|\leq \max_{i\in I\left(x^{\dagger}\right)}\left\|\omega_{i}\right\|_{Y} \left\|F'\left(x^{\dagger}\right)\left(x-x^{\dagger}\right)\right\|_{Y}.
    \end{equation*}
    Consequently,
    \begin{equation*}
        \sum_{i\in I\left(x^{\dagger}\right)}\left|x_{i}-x_{i}^{\dagger}\right| \leq \left|I\left(x^{\dagger}\right)\right|\max_{i\in I\left(x^{\dagger}\right)} \left\|\omega_{i}\right\|_{Y} \left\|F'\left(x^{\dagger}\right)\left(x-x^{\dagger}\right)\right\|_{Y}.
    \end{equation*} 
    Then, by \eqref{equ3.23}, we find that
    \begin{equation}\label{equ3.27}
        \sum_{i\in I\left(x^{\dagger}\right)}\left|x_{i}-x_{i}^{\dagger}\right| \leq \left|I\left(x^{\dagger}\right)\right|\max_{i\in I\left(x^{\dagger}\right)} \left\|\omega_{i}\right\|_{Y} \left(1+\gamma\right)\left\|F\left(x\right)-F\left(x^{\dagger}\right)\right\|_{Y}.
    \end{equation}
    By combining \eqref{equ3.26} and \eqref{equ3.27}, we obtain
    \begin{align}\label{equ3.28}
        \left(1-\eta\right)\left\|x-x^{\dagger}\right\|_{\ell_{1}}
        & \leq \frac{\mathcal{R}_{\eta}\left(x\right)-\mathcal{R}_{\eta}\left(x^{\dagger}\right)} {\left\|x^{\dagger}\right\|_{\ell_{1}}} \nonumber\\
        & \quad +\left(2+\frac{M_{1}}{\left\|x^{\dagger}\right\|_{\ell_{1}}}\eta\right) \left|I\left(x^{\dagger}\right)\right|\max_{i\in I\left(x^{\dagger}\right)} \left\|\omega_{i}\right\|_{Y} \left(1+\gamma\right)\left\|F\left(x\right)-F\left(x^{\dagger}\right)\right\|_{Y},
    \end{align}
    i.e.,
    \begin{equation*}
        c_{3}\left(1-\eta\right)\left\|x-x^{\dagger}\right\|_{\ell_{1}}\leq \mathcal{R}_{\eta}\left(x\right)-\mathcal{R}_{\eta}\left(x^{\dagger}\right)+ \left(c_{4}-c_{5}\eta\right)\left\|F\left(x\right)- F\left(x^{\dagger}\right)\right\|_{Y},
    \end{equation*}
    where $c_{3}=\left\|x^{\dagger}\right\|_{\ell_{1}}$ and 
    \begin{equation*}
        c_{4}=2c_{3}\left|I\left(x^{\dagger}\right)\right|\max_{i\in I\left(x^{\dagger}\right)} \left\|\omega_{i}\right\|_{Y} \left(1+\gamma\right), \quad c_{5}=-M_{1}\left|I\left(x^{\dagger}\right)\right|\max_{i\in I\left(x^{\dagger}\right)} \left\|\omega_{i}\right\|_{Y} \left(1+\gamma\right).
    \end{equation*}
    It is evident that $c_{4}>c_{5}$ and $c_{4}-c_{5}\eta>0$. The proof is completed.
\end{proof}

\par Note that the proof of Lemma 3.7 in \cite{LD25} can be completed up to (3.8) in \cite{LD25}, where (3.8) corresponds to this estimate by choosing $M_{1}$ as the right-hand side of \eqref{equ3.9}. With Lemma \ref{lem3.10} available, we can derive the following result, with a proof that closely follows that of \cite[Theorem 2.18]{LD25}.

\begin{theorem}\label{the3.11}
    Suppose Assumption \ref{ass3.8} is satisfied. Let $x_{\alpha,\eta}^{\delta}$ be defined by \eqref{equ2.1}, and let the constants $c_{3}>0$ and $c_{4}>c_{5}$ be as stated in Lemma \ref{lem3.10}. \\
    (1) If $q=1$ and $1-\alpha\left(c_{4}-c_{5}\eta\right)>0$, then
    \begin{equation}\label{equ3.29}
        \left\|x_{\alpha,\eta}^{\delta}-x^{\dagger}\right\|_{\ell_{1}}\leq \frac{1+\alpha\left(c_{4}-c_{5}\eta\right)} {c_{3}\alpha\left(1-\eta\right)}\delta ,\quad \left\|F\left(x_{\alpha,\eta}^{\delta}\right)-y^{\delta}\right\|_{Y}\leq \frac{1+\alpha\left(c_{4}-c_{5}\eta\right)}{1- \alpha\left(c_{4}-c_{5}\eta\right)}\delta.
    \end{equation}
    (2) If $q>1$, then
    \begin{align}\label{equ3.30}
        & \left\|x_{\alpha,\eta}^{\delta}-x^{\dagger}\right\|_{\ell_{1}} \leq \frac{1}{c_{3}\alpha\left(1-\eta\right)} \left[\frac{\delta^{q}}{q}+\alpha\left(c_{4}-c_{5}\eta\right)\delta+ \frac{\left(q-1\right)2^{\frac{1}{q-1}} \alpha^{\frac{q}{q-1}} \left(c_{1}-c_{2}\eta\right)^{\frac{q}{q-1}}}{q}\right], \nonumber\\ 
        & \left\|F\left(x_{\alpha,\eta}^{\delta}\right)-y^{\delta}\right\|_{Y}^{q} \leq q  \left[\frac{\delta^{q}}{q}+\alpha\left(c_{4}-c_{5}\eta\right)\delta+ \frac{\left(q-1\right)2^{\frac{1}{q-1}} \alpha^{\frac{q}{q-1}} \left(c_{4}-c_{5}\eta\right)^{\frac{q}{q-1}}}{q}\right].
    \end{align}
\end{theorem}

\par Obviously, if $\alpha\sim\delta^{q-1}$ with $q>1$, then $\left\|x_{\alpha,\eta}^{\delta}-x^{\dagger}\right\|_{\ell_{2}}\leq c\delta$ for some constant $c>0$. In the specific case when $q=1$ and $\alpha\sim\delta^{1-\epsilon}$ where $(0<\epsilon<1)$, we have $\left\|x_{\alpha,\eta}^{\delta}-x^{\dagger}\right\|_{\ell_{2}}\leq c\delta^{1-\epsilon}$ for some constant $c>0$, where $0<\eta< 1$.

\par It is essential to note that we cannot obtain the inequality \eqref{equ3.25} if the restriction \eqref{equ3.23} is replaced by \eqref{equ3.2}. Moreover, the results above regarding the convergence rate are valid only for the case $0<\eta< 1$. When $\eta=1$, Lemma \ref{lem3.4} and Lemma \ref{lem3.10} lose their relevance, rendering the proofs of the convergence rate invalid.

\par If the regularization parameter $\alpha$ is determined by Morozov's discrepancy principle, establish the convergence rate $\mathcal{O}\left(\delta\right)$, as demonstrated in \cite[Theorem 3.10]{LD25}.

\begin{theorem}\label{the3.12}
    \textbf{(Discrepancy principle)} Maintain the assumptions of Lemma \ref{lem3.10} and let $x_{\alpha,\eta}^{\delta}$ be defined by \eqref{equ2.1}. If there exist parameters $\alpha>0$ and $0<\eta<1$ such that
    \begin{equation*}
        \delta\leq\left\|F\left(x_{\alpha,\eta}^{\delta}\right)-y^{\delta}\right\|_{Y} \leq \tau\delta,\quad \tau\geq 1.
    \end{equation*}
    Then it follows that
    \begin{equation*}
        \left\|x_{\alpha,\eta}^{\delta}-x^{\dagger}\right\|_{\ell_{2}}\leq \frac{\left(c_{4}-c_{5}\eta\right)\left(\tau+1\right)\delta}{1-\eta}.
    \end{equation*}
\end{theorem}

\section{Computational approach}\label{sec4}

\par In this section, we propose a half-variation algorithm based on the proximal gradient method to address problem \eqref{equ1.2}. We demonstrate the convergence of the algorithm and show that the proximal gradient method can be utilized for the $\left\|x\right\|_{\ell_{1}}^{2}-\eta\left\|x\right\|_{\ell_{2}}^{2}$ sparsity regularization where $\left(0<\eta\leq 1\right)$, specifically for nonlinear inverse problems.

\subsection{Proximal gradient method}\label{sec4.1}

\par We begin with a brief overview of the proximal gradient method. The foundation for this discussion is based on \cite[Chapter 10]{B17}, which proposes a proximal gradient method for solving minimization problems of the form
\begin{equation}\label{equ4.1}
    \min_{x\in X}\left\{f\left(x\right)+g\left(x\right)\right\},
\end{equation}
where $X$ represents a Hilbert space and we assume the following.

\begin{assumption}\label{ass4.1}
    1. $g: X\rightarrow (-\infty, \infty]$ is proper closed and convex.
    \item 2. $f: X\rightarrow (-\infty, \infty]$ is proper and closed. ${\rm dom}\left(g\right) \subset {\rm int}\left({\rm dom}\left(f\right)\right)$, and the gradient of $f\left(x\right)$ is $L_{f}$-Lipschits continuous over ${\rm int}\left({\rm dom}\left(f\right)\right)$.
    \item 3. The optimal set of problem \eqref{equ4.1} is nonempty.
\end{assumption}

\par For the solution of problem \eqref{equ4.1}, it is natural to generalize the above idea and to define the next iterate as the minimizer of the sum of the linearization of $f$ around $x^k$, the non-smooth function $g$, and a quadratic prox term:
\begin{equation}\label{equ4.2}
    x^{k+1}=\arg\min_{x\in X}\left\{f\left(x^{k}\right)+\left<\nabla f\left(x^{k}\right), x-x^{k}\right>+g\left(x^{k}\right)+\frac{1}{2t_{k}}\left\|x-x^{k}\right\|^{2}\right\},
\end{equation}
where $t^{k}$ is the stepsize at iteration $k$. After some simple algebraic manipulation and cancellation of constant terms, \eqref{equ4.2} can be rewritten as
\begin{equation*}
    x^{k+1}=\arg\min_{x\in X}\left\{t_{k}g\left(x^{k}\right)+ \frac{1}{2}\left\|x-\left(x^{k}-t_{k}\nabla f\left(x^{k}\right)\right)\right\|^{2}\right\},
\end{equation*}
which by the definition of the proximal operator is the same as
\begin{equation*}
    x^{k+1}={\rm prox}_{t_{k}g}\left(x^{k}-t_{k}\nabla f\left(x^{k}\right)\right).
\end{equation*}

\par The proximal gradient method with taking the step-sizes $t_{k}=\frac{1}{L_{k}}$ with $L_{k}\in\left(\frac{L_{f}}{2}, +\infty\right)$ is initiated in the form of Algorithm \ref{alg1}. 

\begin{algorithm}
\caption{Proximal gradient method}
\begin{algorithmic}\label{alg1}
\STATE{\textbf{Initialization}: Set $k=0$, $x_{0}\in {\rm int}\left(X\right)$.}
\STATE{\textbf{General step}: for any $k=0,1,2,\cdots$ execute the following steps:}
\STATE{\qquad 1: pick $L_{k}\in\left(\frac{L_{f}}{2}, +\infty\right)$ for $L_{f}>0$;}
\STATE{\qquad 2: set 
\begin{equation*}
    x^{k+1}={\rm prox}_{t_{k}g}\left(x^{k}-\frac{1}{L_{k}}\nabla f\left(x^{k}\right)\right);
\end{equation*}
\STATE{\qquad 3: check the stopping criterion and return $x^{k+1}$ as a solution.}}
\end{algorithmic}
\end{algorithm}

\par To enhance the clarity of the article, we recall some results regarding the proximal operator of the squared $\ell_{1}$-norm, as found in \cite{B17}, along with the definition of the half variation (HV) function, as detailed in \cite{LD25}. The expression $\left\|x\right\|_{\ell_{1}}^{2}$ can be framed as the optimal value of an optimization problem characterized by the function
\begin{align}\label{equ4.3}
    \varphi\left(s,t\right)=
    \left\{             
             \begin{array}{ll}
             \frac{s^{2}}{t}, & if \ t>0, \\
             0, & if \ s=0,\ t=0, \\
             \infty, & else. 
             \end{array}
    \right.
\end{align}
We recommend consulting \cite[Section 6.8.2]{B17} for readers seeking detailed proof of the following lemmas.

\begin{lemma}\label{lem4.2}
    \textbf{(Variational representation of $\left\|\cdot\right\|_{\ell_{1}}^{2}$)} For a Hilbert space $X$ and any $x\in X$, the following holds:
    \begin{equation}\label{equ4.4}
        \min_{\lambda}\left\{\sum_{i=1}^{n}\varphi\left(x_{i},\lambda_{i} \right)\right\}=\left\|x\right\|_{\ell_{1}}^{2},
    \end{equation}
    where $\varphi$ is defined in \eqref{equ4.3}. An optimal solution $\overline{\lambda}$ to the minimization problem outlined in \eqref{equ4.4} is given by
    \begin{align*}
        \overline{\lambda}_{i}=
        \left\{             
             \begin{array}{ll}
             \frac{\left|x_{i}\right|}{\left\|x\right\|_{\ell_{1}}}, & if \ x\neq 0, \\
             \frac{1}{n}, & if \ x=0,
             \end{array}
        \right.
        i=1,2,\cdots,n. 
    \end{align*}
\end{lemma}

\begin{lemma}\label{lem4.3}
    \textbf{(Proximal operator of $\left\|\cdot\right\|_{\ell_{1}}^{2}$)} Let $g:X\rightarrow \mathcal{R}$ be given by $g\left(x\right)=\left\|x\right\|_{\ell_{1}}^{2}$, and let $\alpha>0$. Then we have 
    \begin{equation*}
        {\rm prox}_{\alpha g}\left(x\right)=
        \left\{             
             \begin{array}{ll}
             \left(\frac{\lambda_{i}x_{i}}{\lambda_{i}+2\alpha}\right)_{i=1}^{n}, & if \ x\neq 0, \\
             0, & if \ x=0,
             \end{array}
        \right.
    \end{equation*}
    where $\lambda_{i}=\left[\frac{\sqrt{\alpha}\left|x_{i}\right|}{\sqrt{\mu^{*}}}-2\alpha\right]_{+}$ with $\mu^{*}$ being any positive root of the nonincreasing function 
    \begin{equation*}
        \psi\left(\mu\right)=\sum_{i=1}^{n}\left[\frac{\sqrt{\alpha}\left|x_{i}\right|} {\sqrt{\mu}}-2\alpha\right]_{+}-1.
    \end{equation*}
\end{lemma}

\par Note that for $a\in\mathbb{R}$, if $a>0$, $\left[a\right]_{+}=a$; if $a\leq 0$, $\left[a\right]_{+}=0$ in Lemma \ref{lem4.3}. We define the half variation function based on Lemma \ref{lem4.2} and Lemma \ref{lem4.3}.

\begin{definition}\label{def4.4}
    For $x\in X$ and $\alpha>0$, the function $\mathbb{H}_{\alpha}$ is defined as follows
\begin{align}\label{equ4.5}
    \mathbb{H}_{\alpha}(x)=
    \left\{             
             \begin{array}{ll}
             \sum_{i=1}^{n}\frac{\lambda_{i}}{2\alpha+\lambda_{i}}\left<x, e_{i}\right>e_{i}, & if \ x\neq 0, \\
             0, & if \ x= 0. \\
             \end{array}
    \right.
\end{align}
where $e_{i}=(\underbrace{0,\cdots,0,1}_{i},0,\cdots)$ and $\lambda_{i}$ is defined in Lemma \ref{lem4.3}. Thus, $\mathbb{H}_{\alpha}$ is referred to as a half variation function mapping $X$ to $X$.
\end{definition}

\par Since $\mathcal{R}_{\eta}\left(x\right)=\left\|x\right\|_{\ell_{1}}^{2}- \eta\left\|x\right\|_{\ell_{2}}^{2}$, where $0<\eta\leq 1$, is non-convex, the proximal gradient method cannot be directly applied to problem \eqref{equ1.2}. We can rewrite $\mathcal{J}_{\alpha,\eta}^{\delta}\left(x\right)$ in \eqref{equ1.2} as 
\begin{equation}\label{equ4.6}
    \mathcal{J}_{\alpha,\eta}^{\delta}\left(x\right)=f\left(x\right)+g\left(x\right),
\end{equation}
where $f\left(x\right)=\frac{1}{2}\left\|F(x)-y^{\delta}\right\|_{Y}^{2} -\alpha\eta\left\|x\right\|_{\ell_{2}}^{2}$ and $g\left(x\right)=\alpha\left\|x\right\|_{\ell_{1}}^{2}$. It is evident that $g\left(x\right)$ is proper and convex. Since $f\left(x\right)$ is continuous, it is also proper and closed. Next, we verify the existence of optimal solution of problem \eqref{equ4.6}.

\par Thus, the minimization problem in \eqref{equ4.2} can be rewritten as
\begin{equation}\label{equ4.7}
    x^{k+1}=\arg\min_{x\in \ell_{2}}\left\{\frac{1}{2}\left\|x-x^{k}+t^{k}\left(F'\left(x^{k}\right)^{*} \left(F'\left(x^{k}\right)-y^{\delta}\right)-2\alpha\eta x^{k}\right)\right\|_{\ell_{2}}^{2}+ \alpha\left\|x\right\|_{\ell_{1}}^{2}\right\}.
\end{equation}
By Lemma \ref{lem4.2}, \eqref{equ4.7} is equivalent to the variational form
\begin{equation}\label{equ4.8}
    x^{k+1}=\arg\min_{x\in \ell_{2}}\left\{\frac{1}{2}\left\|x-x^{k}+t^{k}\left(F'\left(x^{k}\right)^{*} \left(F'\left(x^{k}\right)-y^{\delta}\right)-2\alpha\eta x^{k}\right)\right\|_{\ell_{2}}^{2}+ \alpha\sum_{i=1}^{n}\varphi\left(x_{i},\lambda_{i}\right)\right\}.
\end{equation}
where $\varphi$ is defined in \eqref{equ4.4} and $\lambda_{i}$ is specified in Lemma \ref{lem4.2}. According to Lemma \ref{lem4.3}, a minimizer of \eqref{equ4.8} can be explicitly computed by
\begin{equation*}
    x^{k+1}={\rm prox}_{t_{k}g}\left(x^{k}-t_{k}\left(F'\left(x^{k}\right)^{*} \left(F'\left(x^{k}\right)-y^{\delta}\right)-2\alpha\eta  x^{k}\right)\right), 
\end{equation*}
i.e.,
\begin{equation*}
    x^{k+1}=\mathbb{H}_{\alpha}\left(x^{k}-\frac{1}{L_{k}} \left(F'\left(x^{k}\right)^{*} \left(F'\left(x^{k}\right)-y^{\delta}\right)-2\alpha\eta  x^{k}\right)\right), 
\end{equation*}
where $t^{k}$ is the step-size at iteration $k$ and $L_{k}=\frac{1}{t_{k}}$.

\begin{lemma}\label{lem4.5}
    A minimizer of problem \eqref{equ4.8} is expressed as
    \begin{equation}\label{equ4.9}
        x^{k+1}=\mathbb{H}_{\alpha}\left(x^{k}+\frac{2\alpha\eta}{L_{k}} x^{k}-\frac{1}{L_{k}}F'\left(x^{k}\right)^{*} \left(F'\left(x^{k}\right)-y^{\delta}\right)\right).
    \end{equation}
\end{lemma}

\begin{proof}
    The proof closely follows that of \cite[Lemma 6.70]{B17}. Problem \eqref{equ4.8} can be reformulated as the problem
    \begin{equation}\label{equ4.10}
        x^{k+1}=\arg\min_{x\in \ell_{2}}\left\{\frac{1}{2}\left|x-u\right|^{2}+\alpha\sum_{i}\varphi(x_{i}, \lambda_{i})\right\},
    \end{equation}
    where $u=x^{k}-\frac{1}{L_{k}}\left(F'\left(x^{k}\right)^{*} \left(F'\left(x^{k}\right)-y^{\delta}\right)-2\alpha\eta x^{k}\right)$ for a fixed $x^{k}$. Minimizing first with respect to $x$, we obtain that $x_{i}=\frac{\lambda_{i}u_{i}}{\lambda_{i}+2\alpha}$ so that
    \begin{equation*}
        x^{k+1}_{i}=\sum_{i=1}^{n}\frac{\lambda_{i}}{2\alpha+\lambda_{i}}\left[x^{k}+\frac{2\alpha\eta}{L_{k}} x^{k}-\frac{1}{L_{k}}F'\left(x^{k}\right)^{*} \left(F'\left(x^{k}\right)-y^{\delta}\right)\right]_{i}.
    \end{equation*}
    and we can obtain \eqref{equ4.10}.
\end{proof}

\par In this paper, we propose a numerical algorithm based on the idea that the next iterate is obtained by solving the minimization problem \eqref{equ1.2} with proximal gradient method. We called it HV-$\left(\ell_{1}^{2}-\eta\ell_{2}^{2}\right)$ algorithm which is summarized in Algorithm \ref{alg2}.

\begin{algorithm}
\caption{HV-($\ell_{1}^{2}-\eta\ell_{2}^{2}$) algorithm for problem \eqref{equ1.2}}
\begin{algorithmic}\label{alg2}
\STATE{\textbf{Initialization}: Set $k=0$, choose $x_{0}\in \ell_{2}$ such that $\Phi\left(x^{0}\right)<+\infty$.}
\STATE{\textbf{General step}: for any $k=0,1,2,\cdots$ execute the following steps:}
\STATE{\qquad 1: pick $L_{k} =L \in \left(\frac{L_{f}}{2},\infty\right)$ for $L_{f}>0$;}
\STATE{\qquad 2: 
\begin{equation*}
    x^{k+1}=\mathbb{H}_{\alpha}\left(x^{k}+ \frac{2\alpha\eta}{L_{k}} x^{k}- \frac{1}{L_{k}}F'\left(x^{k}\right)^{*}\left(F\left(x^{k}\right)- y^{\delta}\right)\right);
\end{equation*}}
\STATE{\qquad 3: check the stopping criterion and return $x^{k+1}$ as a solution.}
\end{algorithmic}
\end{algorithm}

\subsection{Convergence analysis}

\par Since the gradient of $f\left(x\right)$ is $L_{f}$-Lipschits continuous, and $g\left(x\right)$ is proper convex, we obtain the following results directly, detailed proof refer to \cite[Lemma 5.7 and Theorem 6.39]{B17}.

\begin{lemma}\label{lem4.6}
    Let the gradient of $f:\ell_{2}\rightarrow (-\infty,\infty]$ be $L_{f}$-Lipschits continuous over a given convex set $\mathcal{C}$, i.e.
    \begin{equation}\label{equ4.11}
        \left\|\nabla f\left(x_{1}\right)-\nabla f\left(x_{2}\right)\right\|\leq L_{f}\left\|x_{1}-x_{2}\right\|
    \end{equation}
    for any $x_{1},x_{2}\in\mathcal{C}$, then we have
    \begin{equation}\label{equ4.12}
        f\left(x_{1}\right)\leq f\left(x_{2}\right)+\left< \nabla f\left(x_{2}\right),x_{1}-x_{2}\right>+ \frac{L_{f}}{2}\left\|x_{1}-x_{2}\right\|^{2}.
    \end{equation}
\end{lemma}

\begin{lemma}\label{lem4.9}
    Let $g: \ell_{2} \rightarrow (-\infty,\infty]$ be proper closed and convex. If $u_{1}={\rm prox}_{g}\left(x_{1}\right)$ for any $x_{1},u_{1}\in X$, then we have
    \begin{equation}\label{equ4.14}
        \left<x_{1}-u_{1}, x_{2}-u_{1}\right>\leq g\left(x\right)- g\left(u\right)
    \end{equation}
    for any $x\in X$.
\end{lemma}

\par To reduce notational redundancy and ensure the rigor of subsequent derivations, we set 
\begin{equation*} 
    T_{L}^{g}\left(x\right)={\rm prox}_{\frac{1}{L}g}\left(x-\frac{1}{L}\nabla f\left(x\right)\right)={\rm prox}_{\frac{1}{L}g}\left(x- \frac{1}{L}F'\left(x\right)^{*}\left(F\left(x\right)- y^{\delta}\right)+ \frac{2\alpha\eta}{L} x\right)
\end{equation*}
and 
\begin{equation*} 
    G_{L}^{g}\left(x\right)=L\left(x-T_{L}^{g}\left(x\right)\right)
\end{equation*}
for a constant $L\in \left(\frac{L_{f}}{2},\infty\right)$ in the remaining part of the proof. 

\begin{lemma}\label{lem4.10}
    Suppose that $f$ and $g$ satisfy the conditions in Assumption \ref{ass4.1}. Then for any $x\in {\rm int}\left(f\right)$ and $L\in \left(\frac{L_{f}}{2},\infty\right)$, the following inequality holds
    \begin{equation}\label{equ4.15}
        \mathcal{J}_{\alpha,\eta}^{\delta}\left(x\right)- \mathcal{J}_{\alpha,\eta}^{\delta}\left(T_{L}^{g}\left(x\right)\right) \geq \frac{L-L_{f}/2}{L^{2}}\left\|G_{L}^{g}\left(x\right)\right\|^{2}.
    \end{equation}
\end{lemma}

\begin{proof}
    By Lemma \ref{lem4.6}, we have
    \begin{equation}\label{equ4.16}
        f\left(T_{L}^{g}\left(x\right)\right)\leq f\left(x\right)+\left< \nabla f\left(x\right),T_{L}^{g}\left(x\right)-x\right>+ \frac{L_{f}}{2}\left\|T_{L}^{g}\left(x\right)-x\right\|^{2}.
    \end{equation}
    Furthermore, by Lemma \ref{lem4.9}, we obtain
    \begin{equation*}
        \left<x-\frac{1}{L}\nabla f\left(x\right)-T_{L}^{g}\left(x\right), x-T_{L}^{g}\left(x\right)\right>\leq \frac{1}{L}g\left(x\right)- \frac{1}{L}g\left(T_{L}^{g}\left(x\right)\right).
    \end{equation*}
    It follows that
    \begin{equation*}
        \left<\nabla f\left(x\right), T_{L}^{g}\left(x\right)-x\right>\leq g\left(x\right)-g\left(T_{L}^{g}\left(x\right)\right) -L\left\|T_{L}^{g}\left(x\right)-x\right\|^{2}.
    \end{equation*}
    Combining with \eqref{equ4.16} yields
    \begin{equation*}
        f\left(T_{L}^{g}\left(x\right)\right)+g\left(T_{L}^{g}\left(x\right)\right)\leq f\left(x\right)+g\left(x\right)- \left(L-\frac{L_{f}}{2}\right)\left\|T_{L}^{g}\left(x\right)-x\right\|^{2}.
    \end{equation*}
    By taking into account the definition of $G_{L}^{g}\left(x\right)$, $\mathcal{J}_{\alpha,\eta}^{\delta}\left(x\right)=f\left(x\right)+g\left(x\right)$ and $\mathcal{J}_{\alpha,\eta}^{\delta}\left(T_{L}^{g}\left(x\right)\right)= f\left(T_{L}^{g}\left(x\right)\right)+g\left(T_{L}^{g}\left(x\right)\right)$, \eqref{equ4.15} can be obtained.
\end{proof}

\begin{lemma}\label{lem4.11}
    Let $f$ and $g$ satisfy the conditions in Assumption \ref{ass4.1} and $L>0$. If $g_{0}\left(x\right)\equiv 0$, then $G_{L}^{g_{0}}\left(x\right)=\nabla f\left(x\right)$ for any $x\in {\rm int}\left({\rm dom}\left(f\right)\right)$. Furthermore, for $x^{*}\in {\rm int}\left({\rm dom}\left(f\right)\right)$, it holds that $G_{L}^{g}\left(x^{*}\right)=0$ if and only if $x^{*}$ is a stationary point of problem \eqref{equ4.1}.
\end{lemma}

\begin{proof}
    Since ${\rm prox}_{\frac{1}{L}g_{0}}\left(x\right)=x$ for all $x\in\ell_{2}$, it follows that
    \begin{align*}
        G_{L}^{g_{0}}\left(x\right)
        & = L\left(x-T_{L}^{g_{0}}\left(x\right)\right)=  L\left(x-{\rm prox}_{\frac{1}{L}g}\left(x-\frac{1}{L}\nabla f\left(x\right)\right)\right) \\
        & = L\left(x-\left(x-\frac{1}{L}\nabla f\left(x\right)\right)\right) = \nabla f\left(x\right).
    \end{align*}
     Furthermore, $G_{L}^{g}\left(x^{*}\right)=0$ if and only if 
    \begin{equation}\label{equ4.17}
        x^{*}={\rm prox}_{\frac{1}{L}g}\left(x^{*}-\frac{1}{L}\nabla f\left(x^{*}\right)\right)
    \end{equation}
    By the proof of Lemma \ref{lem4.9}, \eqref{equ4.17} holds if and only if
    \begin{equation*}
        x^{*}-\frac{1}{L}\nabla f\left(x^{*}\right)-x^{*}\in \frac{1}{L}\partial g\left(x^{*}\right),
    \end{equation*}
    i.e.,
    \begin{equation*}
        -\nabla f\left(x^{*}\right)\in \partial g\left(x^{*}\right),
    \end{equation*}
    which is the condition for stationary.
\end{proof}

\par Next, we introduce the non-expansiveness of the proximal operator, see \cite[Theorem 6.42]{B17} for details.

\begin{lemma}\label{lem4.12}
    Let $g$ be a proper closed and convex function. Then for any $x_{1},x_{2}\in \ell_{2}$, we have 
    \begin{equation*}
        \left\|{\rm prox}_{g}\left(x_{1}\right)-{\rm prox}_{g}\left(x_{2}\right)\right\| \leq \left\|x_{1}-x_{2}\right\|.
    \end{equation*}
\end{lemma}

\begin{lemma}\label{lem4.13}
    Let $f$ and $g$ satisfy the conditions in Assumption \ref{ass4.1}. Then for any $x_{1},x_{2}\in {\rm int}\left(f\right)$, we have 
    \begin{equation*}
        \left\|G_{L}^{g}\left(x_{1}\right)-G_{L}^{g}\left(x_{2}\right)\right\| \leq \left(2L+L_{f}\right)\left\|x_{1}-x_{2}\right\|.
    \end{equation*}
\end{lemma}

\begin{proof}
    By Lemma \ref{lem4.12}, for any $x_{1},x_{2}\in {\rm dom}\left(f\right)$,
    \begin{align*}
        G_{L}^{g}\left(x_{1}\right)-G_{L}^{g}\left(x_{2}\right)
        & = L\left\|x_{1}-{\rm prox}_{\frac{1}{L}g}\left(x_{1}-\frac{1}{L}\nabla f\left(x_{1}\right)\right) -x_{2}+{\rm prox}_{\frac{1}{L}g}\left(x_{2}- \frac{1}{L}\nabla f\left(x_{2}\right)\right)\right\| \\
        & \leq L\left\|x_{1}-x_{2}\right\|+L\left\|{\rm prox}_{\frac{1}{L}g}\left(x_{1}- \frac{1}{L}\nabla f\left(x_{1}\right)\right)-{\rm prox}_{\frac{1}{L}g}\left(x_{2}- \frac{1}{L}\nabla f\left(x_{2}\right)\right)\right\| \\
        & \leq L\left\|x_{1}-x_{2}\right\|+L\left\|\left(x_{1}- \frac{1}{L}\nabla f\left(x_{1}\right)\right)-\left(x_{2}-\frac{1}{L}\nabla f\left(x_{2}\right)\right)\right\| \\
        & \leq 2L\left\|x_{1}-x_{2}\right\|+\left\|\nabla f\left(x_{1}\right)-\nabla f\left(x_{2}\right)\right\| \\
        & \leq \left(2L+L_{f}\right)\left\|x_{1}-x_{2}\right\|.
    \end{align*}
    The lemma is proved.
\end{proof}

\begin{theorem}\label{the4.15}
     Suppose that Assumption \ref{ass4.1} holds and let $\left\{x^{k}\right\}$ be the sequence generated by Algorithm \ref{alg2} for solving problem \eqref{equ4.6} with a constant stepsize $L_{k}=L\in\left(\frac{L_{f}}{2},+\infty\right)$. Then the sequence $\left\{\mathcal{J}_{\alpha,\eta}^{\delta}\left(x^{k}\right)\right\}$ is nonincreasing, and
     \begin{equation*}
         \mathcal{J}_{\alpha,\eta}^{\delta}\left(x^{k+1}\right)< \mathcal{J}_{\alpha,\eta}^{\delta}\left(x^{k}\right)
     \end{equation*}
     if and only if $x^{k}$ is not a stationary point of \eqref{equ4.6}. Additionally, 
     \begin{equation*}
         \left\|x^{k+1}-x^{k}\right\|\rightarrow 0\quad {\rm as}\quad k\rightarrow\infty,
     \end{equation*}
     and all limit points of the sequence $\left\{x^{k}\right\}$ are stationary points of problem \eqref{equ4.6}.
\end{theorem}

\begin{proof}
    By Lemma \ref{lem4.10}, we set $x^{k}=x$ and 
    \begin{equation*}
        x^{k+1}=T_{L}^{g}\left(x^{k}\right)={\rm prox}_{\frac{1}{L}g}\left(x^{k}-\frac{1}{L}\nabla f\left(x^{k}\right)\right),
    \end{equation*}
    then 
    \begin{equation*}
        G_{L}^{g}\left(x^{k}\right)=L\left(x^{k}-x^{k+1}\right).
    \end{equation*}
    For any $k>0$, we obtain
    \begin{equation*}
        \mathcal{J}_{\alpha,\eta}^{\delta}\left(x^{k}\right)- \mathcal{J}_{\alpha,\eta}^{\delta}\left(x^{k+1}\right)\geq M\left\|x^{k}-x^{k+1}\right\|^{2},
    \end{equation*}
     where $M=L-L_{f}/2$, from which it readily follows that 
    \begin{equation*}
        \mathcal{J}_{\alpha,\eta}^{\delta}\left(x^{k}\right)\geq \mathcal{J}_{\alpha,\eta}^{\delta}\left(x^{k+1}\right).
    \end{equation*}
    If $x^{k}$ is not a stationary point of problem \eqref{equ4.6}, then
    \begin{equation*}
        \left\|x^{k}-x^{k+1}\right\|\neq 0\quad {\rm and}\quad \mathcal{J}_{\alpha,\eta}^{\delta}\left(x^{k+1}\right)< \mathcal{J}_{\alpha,\eta}^{\delta}\left(x^{k}\right).
    \end{equation*} 
    If $x^{k}$ is a stationary point of problem \eqref{equ4.6}, then 
    \begin{equation*}
        \left\|x^{k}-x^{k+1}\right\|= 0\quad {\rm and}\quad x^{k+1}=x^{k},
    \end{equation*} 
    i.e., 
    \begin{equation*}
        \mathcal{J}_{\alpha,\eta}^{\delta}\left(x^{k+1}\right)= \mathcal{J}_{\alpha,\eta}^{\delta}\left(x^{k}\right).
    \end{equation*}
    
    \par Additionally, since the sequence $\left\{\mathcal{J}_{\alpha,\eta}^{\delta}\left(x^{k}\right)\right\}$ is nonincreasing and bounded below, it converges. Thus, 
    \begin{equation*}
        \mathcal{J}_{\alpha,\eta}^{\delta}\left(x^{k}\right)- \mathcal{J}_{\alpha,\eta}^{\delta}\left(x^{k+1}\right)\rightarrow 0\quad {\rm as}\quad  k\rightarrow \infty,
    \end{equation*}
    which implies that 
    \begin{equation*}
        \left\|x^{k+1}-x^{k}\right\|\rightarrow 0\ {\rm as}\ k\rightarrow\infty.
    \end{equation*}
    Let $x^{*}$ be a limit point of $\left\{x^{k}\right\}$. Then there exists a subsequence $\left\{x^{k_{j}}\right\}$ of the sequence $\left\{x^{k}\right\}$ converging to the limit point $x^{*}$. By Lemma \ref{lem4.13}, for any $j\geq 0$,
    \begin{align*}
        \left\|G_{L}^{g}\left(x^{*}\right)\right\|
        & \leq \left\|G_{L}^{g}\left(x^{k_{j}}\right)-G_{L}^{g}\left(x^{*}\right)\right\|+ \left\|G_{L}^{g}\left(x^{k_{j}}\right)\right\| \\
        & \leq \left(2L+L_{f}\right)\left\|x^{k_{j}}-x^{*}\right\| +\left\|G_{L}^{g}\left(x^{k_{j}}\right)\right\|
    \end{align*}
    Since 
    \begin{equation*}
        \left\|x^{k_{j}}-x^{*}\right\|\rightarrow 0\quad {\rm as}\quad j\rightarrow\infty
    \end{equation*}
    and 
    \begin{equation*}
        \left\|G_{L}^{g}\left(x^{k_{j}}\right)\right\|=L\left\|x^{k_{j}}x^{k_{j+1}}\right\| \rightarrow 0\quad {\rm as}\quad j\rightarrow\infty,
    \end{equation*}
    it follows that 
    \begin{equation*}
        \left\|G_{L}^{g}\left(x^{*}\right)\right\|=0,
    \end{equation*}
    and Lemma \ref{lem4.11} implies that $x^{*}$ is a stationary point of problem \eqref{equ4.6}.
\end{proof}

\section{Numerical experiment}\label{sec5}


\par In this section, we are interested in the sparse recovery for a CS problem where the observed signal is measured with some nonlinear system \cite{B13, YWLJWJ15}. Fortunately, \cite{B13} shows that if the system satisfies some nonlinear conditions, then recovery should still be possible. Under the nonlinear CS frame, the measurement system is nonlinear. Assume that the observation model is 
\begin{equation}\label{equ5.1}
    y=F\left(x\right)+\delta,
\end{equation}
where $\delta\in\mathbb{R}^{m}$ is a noise level, $x\in\mathbb{R}^{n}$ and $F:\mathbb{R}^{m}\rightarrow\mathbb{R}^{n}$ is a nonlinear operator. It can be shown that if the linearization of $F$ at an exact solution $x^{\dagger}$ satisfies the restricted isometry property (RIP), then the convergence property of the IHTA is guaranteed in \cite{B13}. Next we illustrate the efficiency of the proposed algorithm by a nonlinear CS example of the form
\begin{equation}\label{equ5.2}
    y=F\left(x\right):=\hat{a}\left(A\hat{b}\left(x\right)\right),
\end{equation}
which was introduced in \cite{B13}, where $A$ is a CS matrix, $\hat{a}$ and $\hat{b}$ are nonlinear operators, respectively. Here, $\hat{a}$ encodes nonlinearity after mixing by $A$ as well as nonlinear "crosstalk" between mixed elements, and $\hat{b}$ encodes the same system properties for the inputs before mixing. For simplicity, we write $\hat{a}=x+a\left(x\right)$ and $\hat{b}=x+b\left(x\right)$, where again $a$ and $b$ are nonlinear maps. In particular, let $a\left(x\right)=x^{c}$ and $b\left(x\right)=x^{d}$, where $c,d\in\mathbb{N}^{+}$ and $x^{c}$ and $x^{d}$ should be understood in a componentwise sense.

\par Next, we consider the nonlinearity of the operator $F$. In \cite{B13}, it is shown that Jacobian matrix of $\hat{a}\left(A\hat{b}\left(x\right)\right)$ is of the form
\begin{equation*}
    F'\left(x\right)=\left(I+a'\right)\left[A\left(I+b'\right)\right]\left(x\right).
\end{equation*}
We assume that $\left\|x\right\|_{\ell_{2}}$ is bounded, then $a'$ is bounded on bounded sets. Then there exists a constant $c>0$ such that $\left\|f'\left(x_{1}\right)- f'\left(x_{2}\right)\right\|_{L\left(\mathbb{R}^{n},\mathbb{R}^{m}\right)}\leq c\left\|x_{1}-x_{2}\right\|_{\ell_{2}}$, i.e. $f'\left(x\right)$ is Lipschitz continuous, cf. \cite[Lamma 3 \& 4]{B13}.

\begin{figure}[htbp]
    \centering
    \subfigure[True signal]{\includegraphics[width=0.45\columnwidth,height=0.3\linewidth]{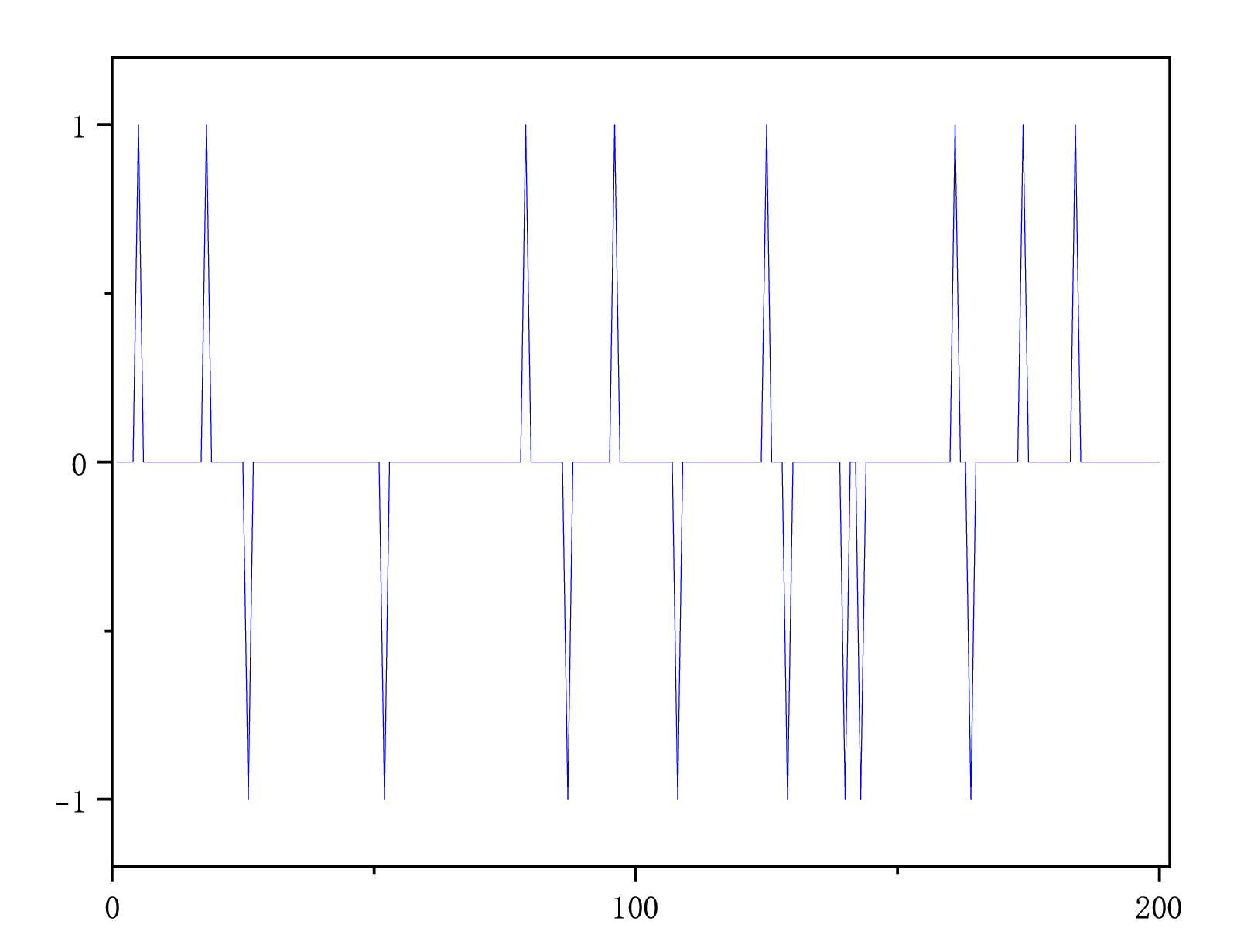}}
    \subfigure[Exact data and Observed data($\sigma$=30dB)]{\includegraphics[width=0.45\columnwidth,height=0.3\linewidth]{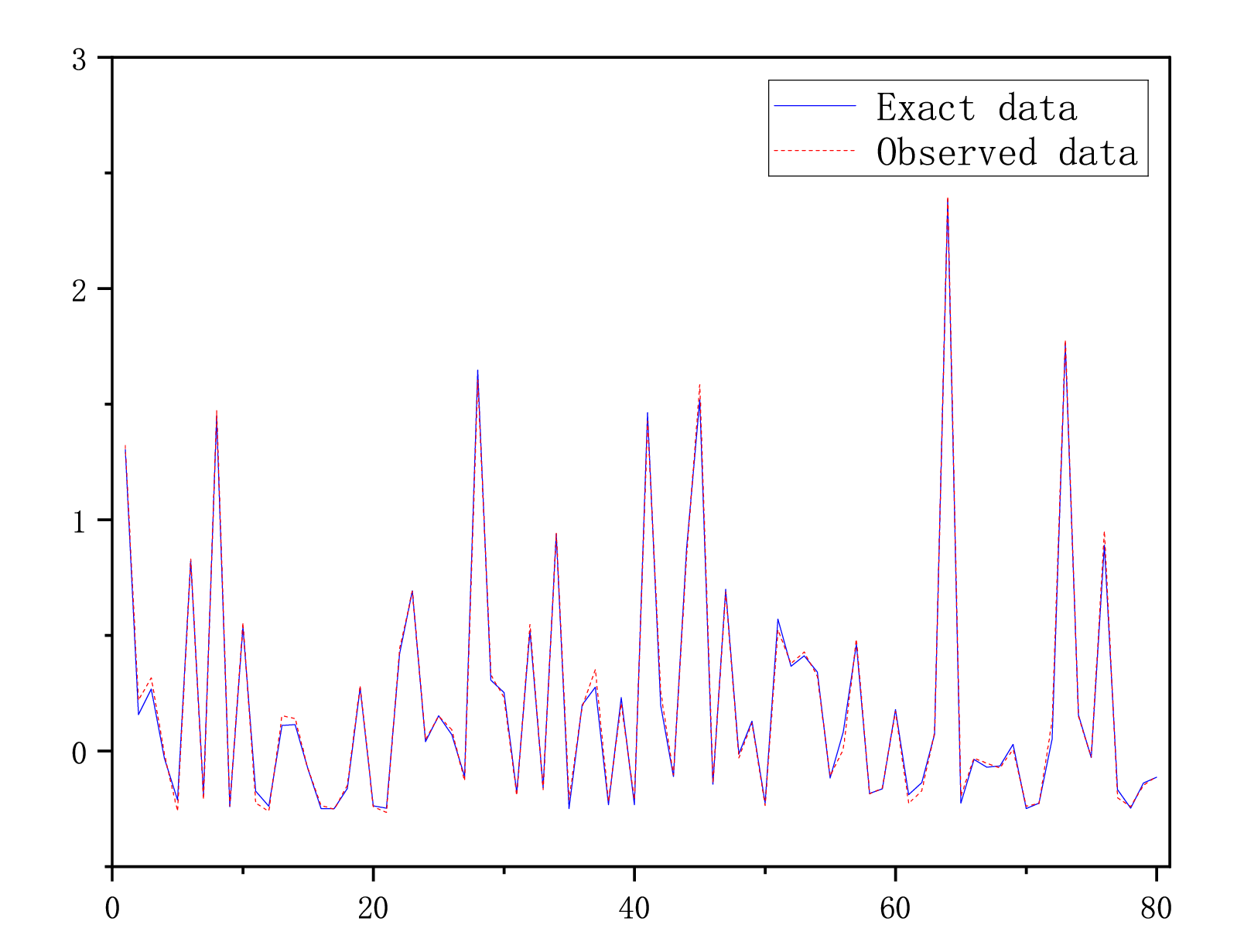}}\\
    \subfigure[HV-$\left(\ell_{1}^{2}-\eta\ell_{2}^{2}\right)$ algorithm with $\eta=0$, {\rm SNR} = 35.3088 ]{\includegraphics[width=0.45\columnwidth,height=0.3\linewidth]{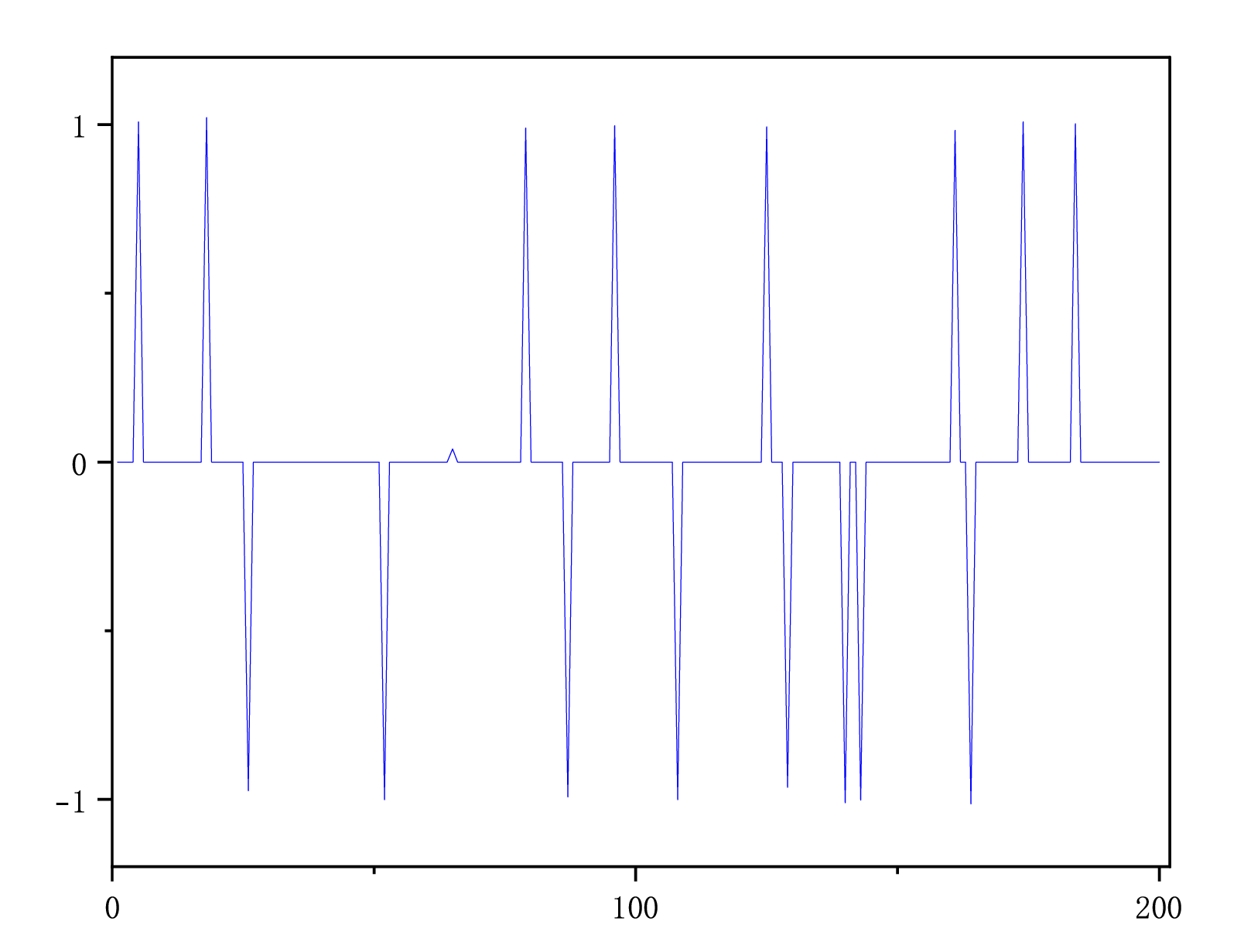}}
    \subfigure[HV-$\left(\ell_{1}^{2}-\eta\ell_{2}^{2}\right)$ algorithm with $\eta=0.4$, {\rm SNR} = 36.1417 ]{\includegraphics[width=0.45\columnwidth,height=0.3\linewidth]{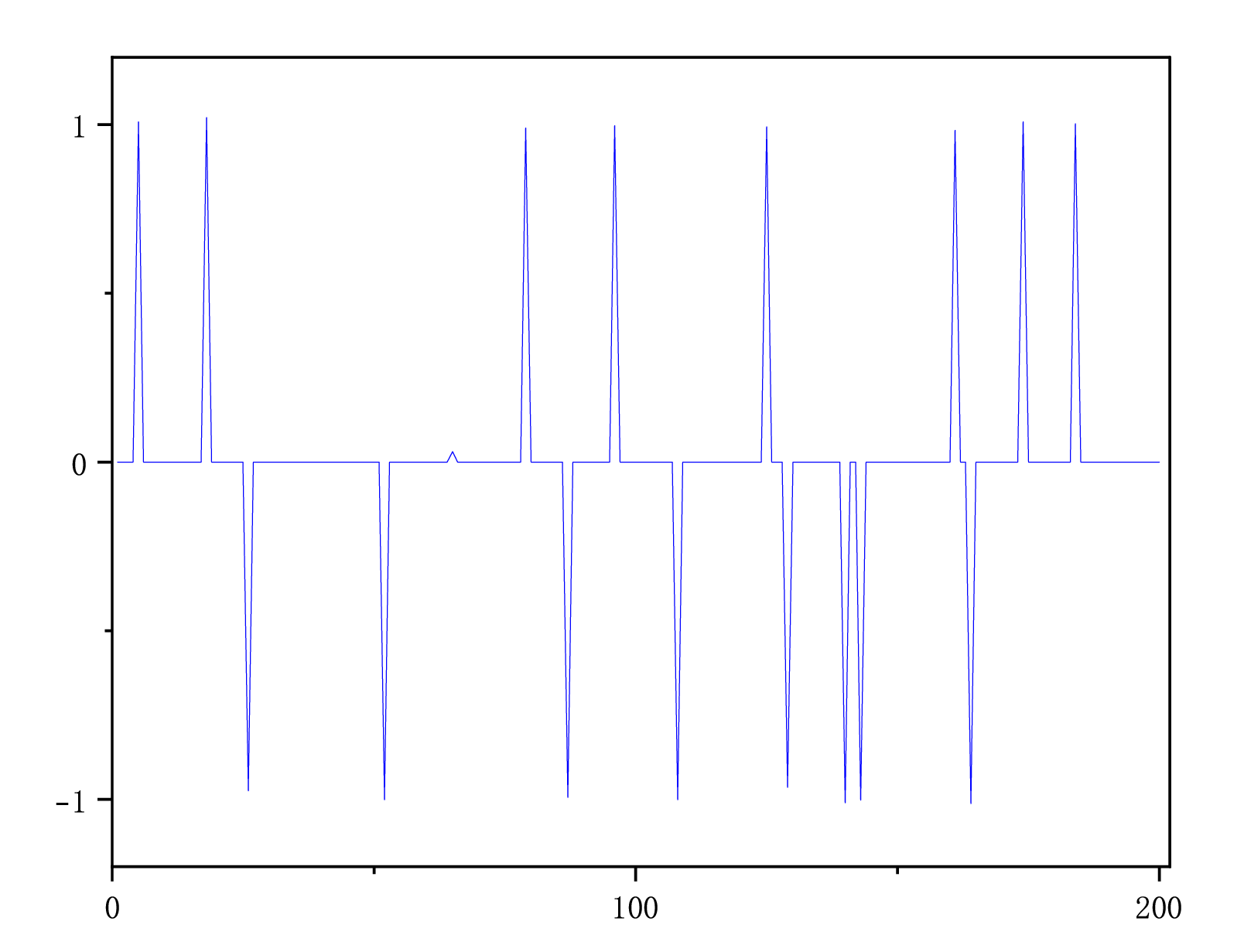}}\\
    \subfigure[HV-$\left(\ell_{1}^{2}-\eta\ell_{2}^{2}\right)$ algorithm with $\eta=0.6$, {\rm SNR} = 36.6058 ]{\includegraphics[width=0.45\columnwidth,height=0.3\linewidth]{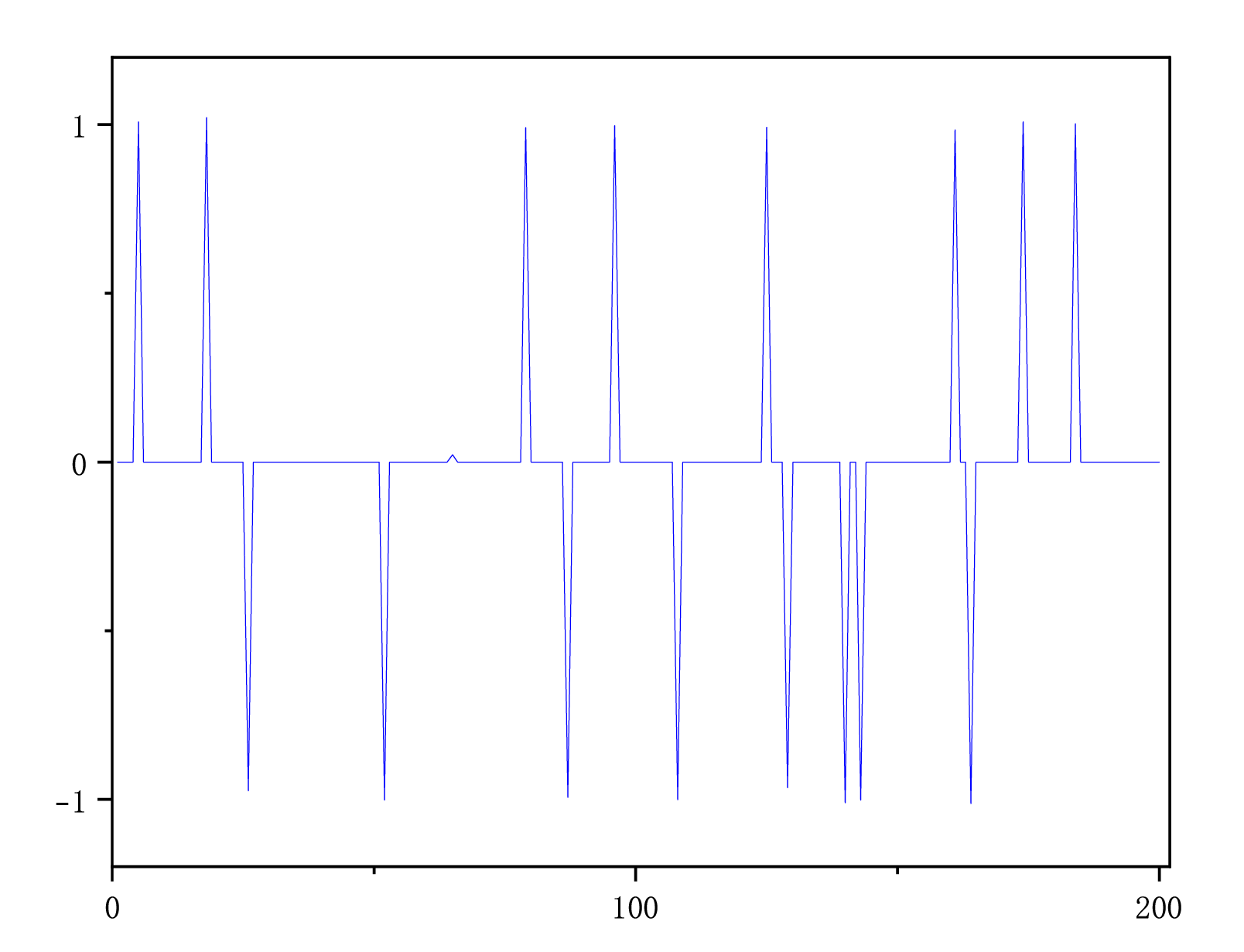}}
    \subfigure[HV-$\left(\ell_{1}^{2}-\eta\ell_{2}^{2}\right)$ algorithm with $\eta=1.0$, {\rm SNR} = 37.4681 ]{\includegraphics[width=0.45\columnwidth,height=0.3\linewidth]{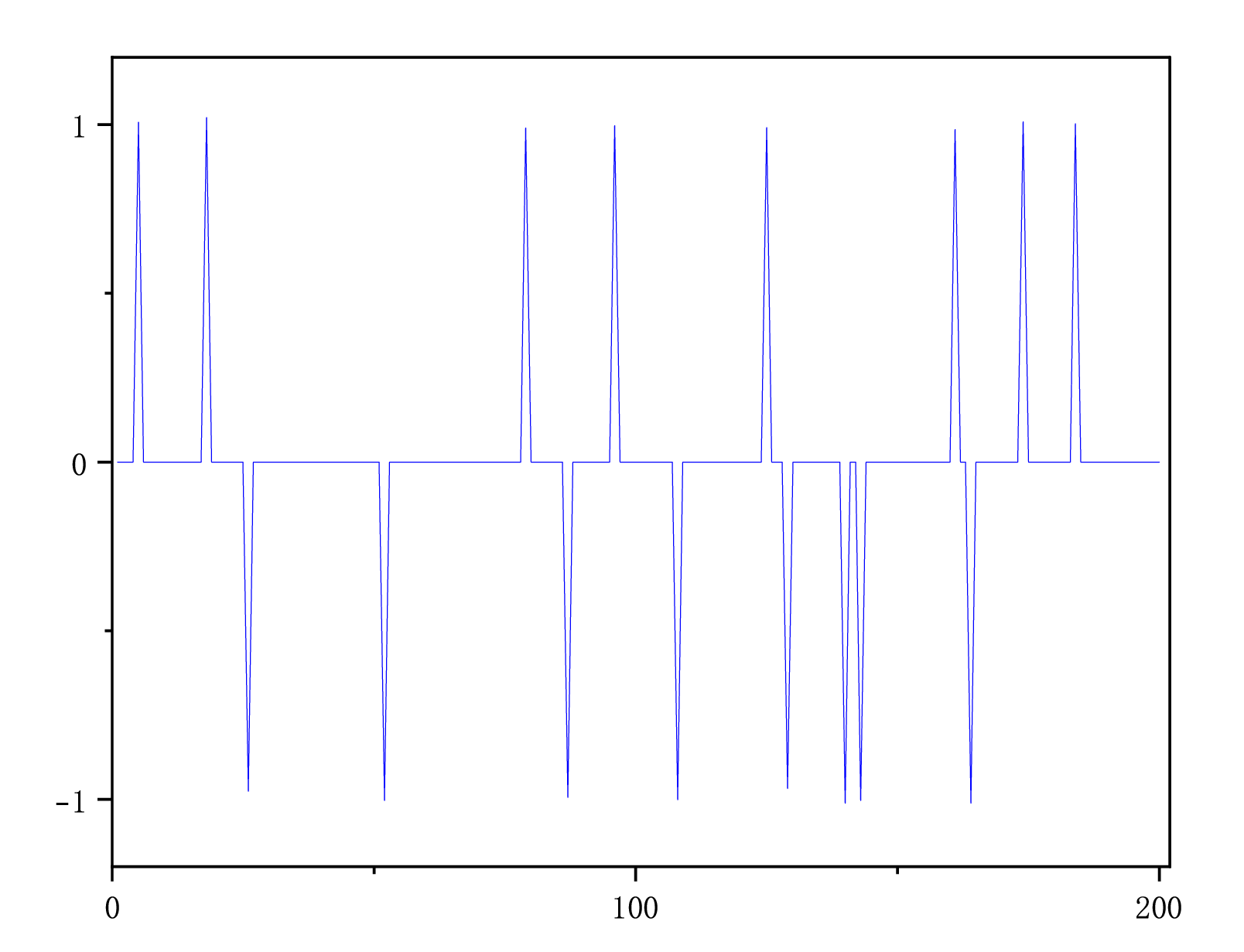}}
    \caption{(a) True signal.\ (b) Exact data and Observed data with $\delta$=30dB. (c)-(f) The reconstructed signal of HV-$\left(\ell_{1}^{2}-\eta\ell_{2}^{2}\right)$ algorithm with different $\eta$.}
    \label{figure:1}
\end{figure}

\par We present several numerical tests which demonstrate the efficiency of the proposed method. To make Algorithm \ref{alg2} clear to the reader, we study the influence of the parameters $\lambda$, $\eta$, and the nonlinear maps $a$ and $b$ on the iterative result $x^{*}$. For the numerical simulation, we use a setting that $A$ is a Gaussian random measurement matrix. The nonlinear CS problem is of the form $\hat{a}\left(A_{m\times n}\hat{b}\right)\left(x_{n}\right)=y_{m}$, where $A_{m\times n}$ is a Gaussian random measurement matrix. The exact solution $x^{\dagger}$ is $s$-sparse. The exact data $y^{\dagger}$ is obtained by $y^{\dagger}=\hat{a}\left(A\hat{b}\left(x^{\dagger}\right)\right)$. White Gaussian noise  is added to the exact data $y^{\dagger}$ and $\delta$ is the noise level, measured in dB. The iterative solution is denoted by $x^{*}$. The performance of the iterative solution $x^{*}$ is evaluated by signal-to-noise ratio (SNR) which is defined by
\begin{equation*}
    {\rm SNR}:= -10{\rm log_{10}}\frac{\left\|x^{*}-x^{\dagger}\right\|_{\ell_{2}}^{2}} {\left\|x^{\dagger}\right\|_{\ell_{2}}^{2}}.
\end{equation*}
We utilize the discrepancy principle to choose the regularization parameter $\alpha$, and the iterative stopping criterion is that the error of the two adjacent iterative solutions is less than $1{\rm E-}5$ or the maximum number of iterations is reached.

\par Let $n=200$, $m=0.4n$, $p=0.2m$, where $p$ is the number of the impulses in the true  solution. For the sparsity regularization of linear ill-posed problems, the value of $\left\|A_{m\times n}\right\|_{\ell_{2}}$ needs to be less than 1 \cite{DDD04}. This requirement is also needed for the nonlinear CS problem \eqref{equ5.1}. The value of $\left\|A_{m\times n}\right\|_{\ell_{2}}$ is around 20, and we rescale the matrix $A_{m\times n}$ by $A_{m\times n}\rightarrow 0.05A_{m\times n}$. The initial value $x^{0}$ in Algorithm \ref{alg2} is generated by calling $x^{0}=0.01*{\rm ones}(n,1)$ in MATLAB. All numerical experiments were conducted using MATLAB R2024a on an AMD Ryzen 5-4500U with Radeon Graphics 2.38GHz workstation with 16GB RAM.

\par Test 1: the impact of $\eta$ on the Algorithm \ref{alg2}.

\par Initially, we examine the convergence and convergence rate of the proposed algorithm. We set $c=2$ and $d=3$, with a noise level of $\delta$ is 30dB. Different values of the parameter $\eta$ are tested to assess their impact on the iterative solution $x^{*}$. Theoretically, for Algorithm \ref{alg2} with a constant $L_{k}=10$. Fig.\ \ref{figure:1}  illustrates the graphs of the iterative solution $x^{*}$  for a regularization parameter $\alpha=5.1{\rm e}$-$5$. The results improve as $\eta$ increases, indicating that non-convex regularization with $\eta>0$ yields superior performance.

\par Test 2: the impact of $L$ on Algorithm \ref{alg2}.

\par Secondly, we explore the influence of the parameter $L$ on the iterative process. In our analysis, we set $n=200$, $m=0.4n$, $p=0.2m$. According to \eqref{equ4.6}, it is evident that the iterative results exhibit only minor variations in response to changes in $L$. We assess the range of values of $L$ that make the gradient of $f(x)$ $L_{f}$-Lipschitz continuous. Our calculations reveal that $\left\|f'\left(x^{\dagger}\right)-f'\left(x^{0}\right)\right\|\approx 22.37$ and $\left\|x^{\dagger}-x^{0}\right\|\approx 4$, which indicates that $L=L_{f}\geq 5.59$. However, empirical observations suggest that an increase in $L$ correlates with a higher number of iterations and longer computation times. In Table \ref{table:1}, we establish $\eta=1.0$ while maintaining the values of $\alpha$, $\delta$, $c$ and $d$ consistent with those in Test 1. We present the iterative results of Algorithm \ref{alg2} with different $L$. Our research findings demonstrate that increasing the value of $L$ initially improves the SNR of the reconstructed signal produced by Algorithm \ref{alg2}, but this improvement is followed by a gradual decrease. 

\par Theoretically, the computational time of the algorithm reaches its lowest point as the parameter approaches 5.59. However, numerical experiments indicate that the SNR of the reconstructed signal significantly deviates from the optimal value at this specific parameter. To determine the optimal value of $L$, it is generally necessary to employ data-driven methods, such as machine learning. This aspect is beyond the purview of the current study and will be addressed in greater detail in forthcoming research endeavors.

\begin{table}[H]
\caption{{\rm SNR} of reconstruction $x^{*}$ with different $L$ when $c=2$, $d=3$ and $\sigma=30$dB.}
\label{table:1}
\centering
\tabcolsep=0.35cm
\tabcolsep=0.01\linewidth
\scalebox{1}{
\begin{tabular}{cccc}
\hline
$L$ \quad & \multicolumn{1}{c}{Iteration number} \quad & \multicolumn{1}{c}{Computation time (s)} \quad & \multicolumn{1}{c}{SNR} \\
\hline 
6  & 660 & 113.9835 & 23.1867 \\
8  & 712 & 120.1261 & 27.6471 \\
10  & 503 & 140.4839 & 37.4681 \\
20  & 1020 & 285.4130 & 37.4247 \\
50  & 2139 & 631.3350 & 37.1039 \\
100  & 2563 & 1101.5282 & 36.5694 \\
\hline
\end{tabular}}
\end{table}

\par Test 3: the impact of $\delta$ on Algorithm \ref{alg2}.

\par Next, we examine the stability of Algorithm \ref{alg2}. To assess the impact of the noise level $\delta$, we introduce several different noise levels to the exact data $y^{\dagger}$. The results of the iterations are presented in Table \ref{table:2}. Notably, the SNR of the inversion solution $x^{*}$ increases as the noise level decreases. It is evident that satisfactory outcomes can be achieved when the noise level $\delta\geq 30$dB. Additionally, Table \ref{table:2} illustrates that Algorithm \ref{alg2} demonstrates good stability in relation to the noise level when the parameter $\eta$ is held constant. This suggests that the constant step-size variant of Algorithm \ref{alg2} is not particularly sensitive to the choice of $\eta$.

\begin{table}[H]
\caption{{\rm SNR} of reconstruction $x^{*}$ with several noise levels when $c=2$ and $d=3$.}
\label{table:2}
\centering
\tabcolsep=0.35cm
\tabcolsep=0.01\linewidth
\scalebox{1}{
\begin{tabular}{ccccccc}
\hline
& \multicolumn{1}{c}{$\eta=0$} & \multicolumn{1}{c}{$\eta=0.2$} & \multicolumn{1}{c}{$\eta=0.4$} & \multicolumn{1}{c}{$\eta=0.6$} & \multicolumn{1}{c}{$\eta=0.8$} & \multicolumn{1}{c}{$\eta=1$} \\
\hline
$\delta=50$dB, $\alpha$=7.4e-6  & 37.6347 & 38.0471 & 38.4366 & 39.0115 & 39.4698 & 39.9806 \\
$\delta=40$dB, $\alpha$=1.2e-5  & 33.5447 & 33.9714 & 34.2901 & 34.7511 & 35.2922 & 35.7130 \\
$\delta=30$dB, $\alpha$=5.1e-5  & 35.3088 & 35.6992 & 36.1416 & 36.6058 & 37.0000 & 37.4681 \\
$\delta=20$dB, $\alpha$=3.0e-4  & 8.4123 & 8.9097 & 9.4336 & 9.7323 & 9.9802 & 10.2690 \\
$\delta=10$dB, $\alpha$=1.9e-3  & 1.5396 & 1.7859 & 2.0118 & 2.3801 & 2.6299 & 2.9761 \\
\hline
\end{tabular}}
\end{table}

\par Test 4:  the impact of the nonlinearity on Algorithm \ref{alg2}.

\par Subsequently, we analyze the effects of the nonlinearity of the function $F$, specifically the parameters $c$ and $d$, on the iterative solution $x^{*}$. The nonlinearity of the CS problem described in equation \eqref{equ5.2} is influenced by the parameters $c$ and $d$. In particular, the degree of nonlinearity in $F$ increases as the parameters $c$ and $d$ rise. In Table \ref{table:3}, we set $\eta=1.0$ and keep $\alpha$, $\delta$, and $s$ consistent with those in Test 1. It is observed that when both $c$ and $d$ are odd numbers, most of the reconstruction results produced by Algorithm \ref{alg2} are relatively good. Conversely, when $d$ is an even number, Algorithm \ref{alg2} struggles, resulting in poor reconstruction outcomes. Additionally, when $d$ is odd and $c$ is even, occasionally a superior reconstruction result may be achieved. In fact, Algorithm \ref{alg2} can only identify the positive impulses and fails to recover negative impulses when $d$ is an even number. Theoretically, due to the non-convexity of $\mathcal{J}_{\alpha,\eta}^{\delta}\left(x\right)$ in \eqref{equ1.2}, the minimizer of $\mathcal{J}_{\alpha,\eta}^{\delta}\left(x\right)$ may be non-unique.

\begin{table}[H]
\caption{{\rm SNR} of reconstruction $x^{*}$ with different parameter $c$ and $d$, iterations number=1000.}
\label{table:3}
\centering
\tabcolsep=0.35cm
\tabcolsep=0.01\linewidth
\scalebox{1}{
\begin{tabular}{cccccccccc}
\hline
& \multicolumn{1}{c}{$d=1$} & \multicolumn{1}{c}{$d=2$} & \multicolumn{1}{c}{$d=3$} & \multicolumn{1}{c}{$d=4$} & \multicolumn{1}{c}{$d=5$} & \multicolumn{1}{c}{$d=6$} & \multicolumn{1}{c}{$d=7$} & \multicolumn{1}{c}{$d=8$} & \multicolumn{1}{c}{$d=9$} \\
\hline
$c=1$  & 30.6587 & 7.2674 & 39.8824 & 5.0509 & 27.2799 & 4.3630 & 13.2113 & 5.6239 & 6.8346 \\
$c=2$  & 4.9926 & 5.0398 & 37.4681 & 5.0488 & 5.3060 & 4.2578 & 6.9915 & 4.3172 & 10.4523 \\
$c=3$  & 30.2730 & 5.0480 & 39.6858 & 6.0195 & 38.0497 & 4.8467 & 13.5417 & 4.6592 & 5.4900 \\
$c=4$  & 5.1510 & 4.9924 & 30.8096 & 6.0180 & 33.0730 & 4.2589 & 35.6891 & 4.5218 & 7.0100 \\
$c=5$  & 25.9771 & 4.2587 & 36.4582 & 4.2589 & 34.0959 & 5.0493 & 10.4374 & 5.7413 & 5.8134 \\
$c=6$  & 12.8677 & 4.2534 & 35.7056 & 4.2552 & 11.0395 & 4.2579 & 4.0870 & 3.5901 & 35.4939 \\
$c=7$  & 21.8889 & 5.0434 & 35.2849 & 4.2579 & 33.9606 & 4.2584 & 39.0054 & 4.2589 & 6.4049 \\
$c=8$  & 14.4911 & 5.0454 & 33.9567 & 4.2585 & 35.9118 & 5.9656 & 40.1391 & 4.2569 & 39.3522 \\
$c=9$  & 20.4868 & 4.2540 & 35.2977 & 5.0497 & 35.5304 & 5.0506 & 39.0892 & 3.5900 & 34.7919 \\
\hline
\end{tabular}}
\end{table}

\par Test 5: Comparison with different algorithms.

\par Finally, to demonstrate the superiority of the algorithm, we compare the reconstruction effects achieved by various algorithms. We maintain the same settings as in Test 1 and evaluate the reconstruction results against those of the iterative soft thresholding algorithm (ISTA) as detailed in \cite{DDD04} and the ST-$\left(\alpha\ell_{1}-\beta\ell_{2}\right)$ algorithm as described in \cite{DH19, DH24}. Fig.\ \ref{figure:3} (b)-(d) display the reconstruction outcomes of the different algorithms along with their corresponding SNR values. The results indicate that HV-$\left(\ell_{1}^{2}-\eta\ell_{2}^{2}\right)$ outperforms the other algorithms. Furthermore, to more thoroughly assess the reconstruction effects of each algorithm, we present the changes in relative error throughout the iterative process over the number of iterations, as illustrated in Fig.\ \ref{figure:4}. The analysis reveals that the descent rate of the relative error (Rerror) for the HV-$(\ell_{1}^{2}-\eta\ell_{2}^{2})$ algorithm exceeds that of both the ISTA and ST-$\left(\alpha\ell_{1}-\beta\ell_{2}\right)$ algorithm throughout the iterative process. This demonstrates the efficacy of the proposed algorithm in addressing sparse inverse problems, highlighting its potential as a robust solution within this domain.

\begin{figure}[htbp]
    \centering
    \subfigure[True signal]{\includegraphics[width=0.45\columnwidth,height=0.3\linewidth]{true_signal.eps}}
    \subfigure[ISTA, {\rm SNR} = 34.9567 ]{\includegraphics[width=0.45\columnwidth,height=0.3\linewidth]{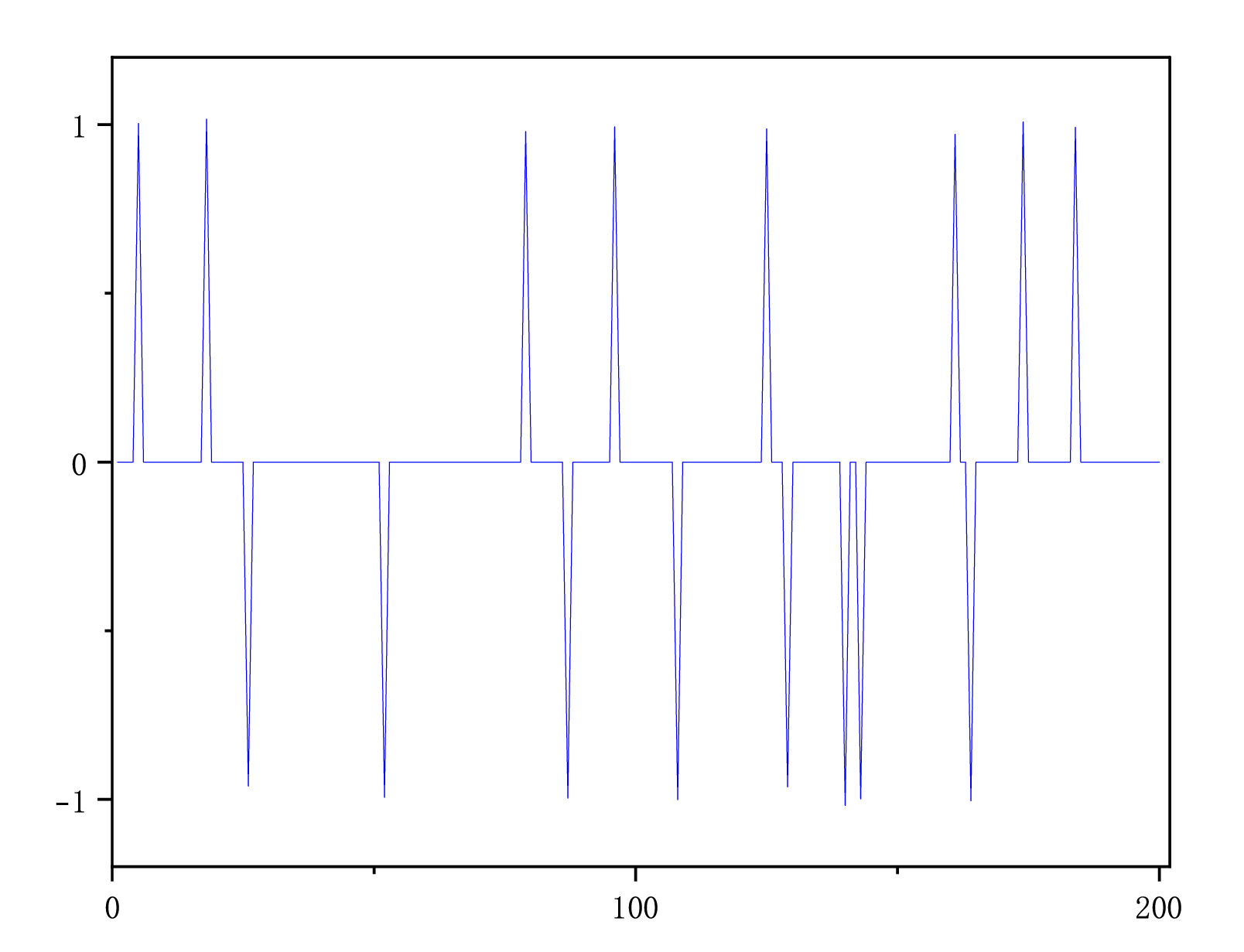}}\\
    \subfigure[ST-$\left(\alpha\ell_{1}-\beta\ell_{2}\right)$ algorithm with $\beta/\alpha=1$, {\rm SNR} = 36.0864 ]{\includegraphics[width=0.45\columnwidth,height=0.3\linewidth]{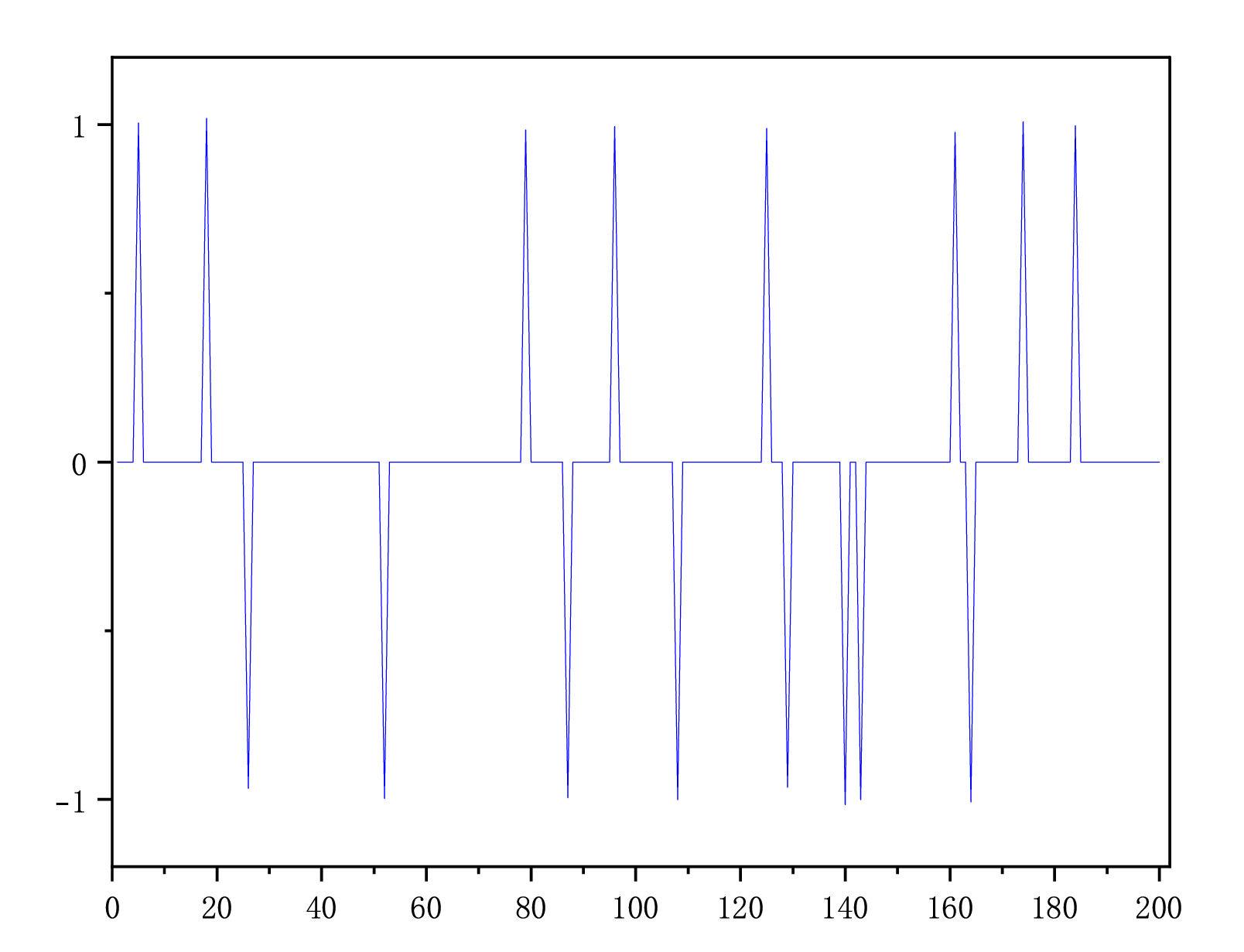}}
    \subfigure[HV-$\left(\ell_{1}^{2}-\eta\ell_{2}^{2}\right)$ algorithm with $\eta=1$, {\rm SNR} = 37.4247 ]{\includegraphics[width=0.45\columnwidth,height=0.3\linewidth]{eta=1.eps}}
    \caption{(a) True signal. (b) The reconstructed signal of ISTA. (c) The reconstructed signal of ST-$\left(\alpha\ell_{1}-\beta\ell_{2}\right)$ algorithm with $\beta/\alpha=1$. 
    (d) The reconstructed signal of HV-$\left(\ell_{1}^{2}-\eta\ell_{2}^{2} \right)$ algorithm with $\eta=1$.}
    \label{figure:3}
\end{figure}

\begin{figure}[htbp]
    \centering
    \includegraphics[width=0.45\columnwidth,height=0.3\linewidth]{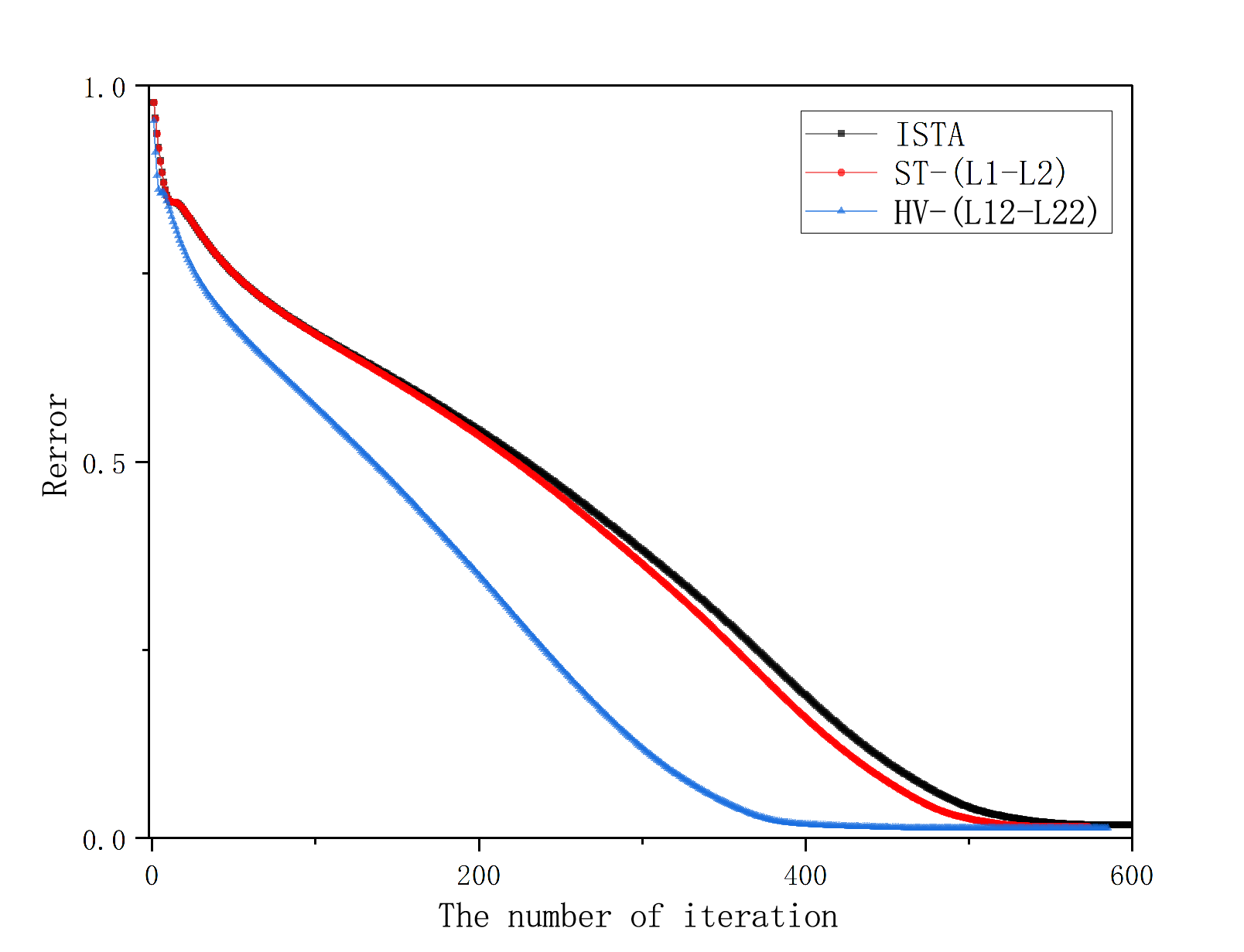}
    \caption{The change curve of Rerror for the ISTA, ST-$\left(\alpha\ell_{1}-\beta\ell_{2}\right)$ and HV-$\left(\ell_{1}^{2}-\eta\ell_{2}^{2}\right)$ algorithms with the increase of iteration.}
    \label{figure:4}
\end{figure}

\section{Conclusion}\label{sec6}

\par We investigated the $\ell_{1}^{2}-\eta\ell_{2}^{2}$ sparsity regularization for nonlinear ill-posed problems, where $(0<\eta\leq 1)$. Regarding the well-posedness of this regularization, in contrast to the scenario where $0<\eta< 1$, we can only establish weak convergence for the case $\eta=1$. Assuming that the nonlinear operator  $F$ satisfies certain conditions, we demonstrated that the regularized solution exhibits sparsity. We identified two distinct convergence rates $\mathcal{O}\left(\delta^{1/2}\right)$ and $\mathcal{O}\left(\delta\right)$, which arise under two commonly accepted nonlinear conditions. Additionally, a half-variation algorithm, denoted as HV-$\left(\ell_{1}^{2}-\eta\ell_{2}^{2}\right)$, can be extended for addressing the non-convex $\ell_{1}^{2}-\eta\ell_{2}^{2}$ $(0<\eta\leq 1)$ sparsity regularization for nonlinear ill-posed problems. Numerical experiments demonstrate that the proposed method is both convergent and stable. In the future, the integration of $\ell_{1}^{2}-\ell_{2}^{2}$ regularization within the projected gradient method holds promise for addressing nonlinear ill-posed problems, potentially enhancing the convergence rate significantly.

\end{document}